\documentclass[10pt]{article}
\usepackage{graphicx}
\usepackage{epsfig}
\usepackage{amssymb}
\usepackage{amsmath}
\usepackage{amsthm}
\newtheorem{lemma}{Lemma}[section]
\newtheorem{theorem}{Theorem}[section]

\newenvironment{remark}{\begin{description} \item[Remark.] }%
{\end{description}}
{\end{description}}

{\hfill $\bullet$ \\}

\newcommand{\be}{\begin{equation}}
\newcommand{\ee}{\end{equation}}

\begin{document}
    \title{A high-order combined interpolation/finite element technique for evolutionary coupled groundwater-surface water problem}
   \author{\Large{Eric Ngondiep, Areej A. Binsultan, Ibtisam M. Aldawish,}
       \thanks{Correspondence to:\ ericngondiep@gmail.com/engondiep@imamu.edu.sa (E. Ngondiep), aabinsultan@imamu.edu.sa (A. A. Binsultant).\ }}
   \date{\small{Department of Mathematics and Statistics, College of Science, Imam Mohammad Ibn Saud Islamic University (IMSIU), 90950 Riyadh, Saudi Arabia}}

    \maketitle
   \textbf{Abstract.}
    A high-order combined interpolation/finite element technique is developed for solving the coupled groundwater-surface water system that governs flows in karst aquifers. In the proposed high-order scheme we approximate the time derivative with piecewise polynomial interpolation of second-order and use the finite element discretization of piecewise polynomials of degree $d$ and $d+1$, where $d \geq 2$ is an integer, to approximate the space derivatives. The stability together with the error estimates of the constructed technique are established in $L^{\infty}(0,T;\text{\,}L^{2})$-norm. The analysis suggests that the developed computational technique is unconditionally stable, temporal second-order accurate and convergence in space of order $d+1$. Furthermore, the new approach is faster and more efficient than a broad range of numerical methods discussed in the literature for the given initial-boundary value problem. Some examples are carried out to confirm the theoretical results.\\
    \text{\,}\\
   \ \noindent {\bf Keywords: evolutionary coupled groundwater-surface water problem, interpolation technique, finite element method, a hybrid higher-order unconditionally stable interpolation/finite element approach, stability analysis and error estimates.} \\
   \\
   {\bf AMS Subject Classification (MSC). 65M10, 65M05}.

      \section{Introduction and preliminaries}\label{sec1}

      \text{\,\,\,\,\,\,\,\,\,\,} The coupled surface flow and porous media flow with interface conditions arises in many applications in applied sciences and engineering including: fluid flow in porous media \cite{7lgzc}, subsurface flow systems \cite{15lgzc,26lgzc}, filtration problems in engineering \cite{37lgzc}, etc... The macroscopic properties of filtration processes are described using the coupled groundwater-surface water model. These processes find several applications in porous media problems such as: oil filtering through rocks/sand and water flowing across semi-permeable soil. The mixed Stokes-Darcy consists of Stokes equations and Darcy's law to govern the flow in the conduits and porous media, respectively, together with the interface conditions to couple both flows \cite{15lgzc,en1,16lgzc,en2}. A large class of numerical methods for solving complex ordinary/partial differential equations (ODEs/PDEs) \cite{en,en4} such as the mixed Stokes-Darcy model have been developed at a moment when it was prohibited to use the Navier-Stokes equations \cite{en3} for modeling many problems because of CPU required and large computer memory. In this paper, we focus on the coupled unsteady groundwater-surface water problem. As literature on numerical methods and applications, interested readers can consult \cite{8en1,7en1,2en1,5en1,3en1,6en1,4en1} and references therein.

       \begin{figure}
         \begin{center}
          \begin{tabular}{c}
          \psfig{file=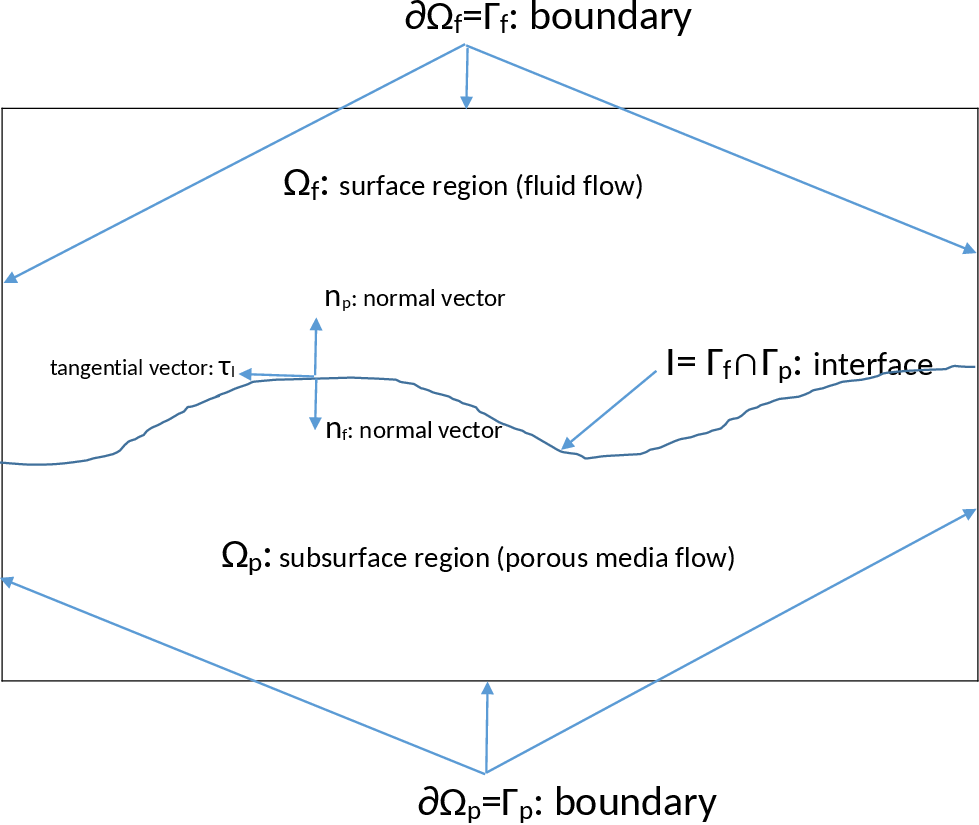,width=8cm}\\
         \end{tabular}
        \end{center}
         \caption{Configuration of a coupled groundwater-surface water flow separated with an interface I}
          \label{fig1}
          \end{figure}

      We consider a fluid flow in domain $\Omega_{f}\subset\mathbb{R}^{m}$ coupled with a porous media flow in region $\Omega_{p}\subset\mathbb{R}^{m}$, where $m=2$ or $3$, and lie across an interface $I$ from each other, where $\Omega_{s}$, $s\in\{f,p\}$, are bounded domains satisfying: $\Omega_{f}\cap\Omega_{p}=\emptyset$ and $\overline{\Omega}_{f}\cap\overline{\Omega}_{p}=I$ (see Figure $\ref{fig1}$). Set $\overline{\Omega}=\overline{\Omega}_{f}\cup\overline{\Omega}_{p}$ and let $n_{f}$ and $n_{p}$ be the unit outward normal vectors on $\Gamma_{f}=\partial\Omega_{f}$ (boundary of $\Omega_{f}$) and $\Gamma_{p}=\partial\Omega_{p}$ (boundary of $\Omega_{p}$), respectively. Denote $\tau_{l},$ $l=1,2,...,m-1,$ be the unit tangential vectors on the interface coupling $I$. We remind that $n_{f}=-n_{p}$.\\

       The fluid velocity $v$, fluid pressure $p$, and hydraulic head $\phi$ satisfy the following equations \cite{9en1}
        \begin{equation}\label{1}
     v_{t}-\nu\Delta v+ \nabla p=f(x,t), \text{\,\,\,\,\,in\,\,\,}\Omega_{f}\times[0,\text{\,}T],
     \end{equation}
     \begin{equation}\label{2}
     \nabla\cdot v = 0, \text{\,\,\,\,\,in\,\,\,}\Omega_{f}\times[0,\text{\,}T],
     \end{equation}
     \begin{equation}\label{3}
     S_{0}\phi_{t}-\nabla\cdot (\mathcal{K}\nabla\phi)=g_{0}(x,t), \text{\,\,\,\,\,in\,\,\,}\Omega_{p}\times[0,\text{\,}T],
     \end{equation}
     \begin{equation}\label{4}
     v_{p} =-\frac{1}{\eta}\mathcal{K}\nabla\phi, \text{\,\,\,\,\,in\,\,\,}\Omega_{p}\times[0,\text{\,}T],
     \end{equation}
     subject to initial conditions
     \begin{equation}\label{5}
      v(x,0)=v_{0}(x), \text{\,\,\,for\,\,}x\in\Omega_{f} \text{\,\,\,\,\,and\,\,\,\,\,} \phi(x,0)=\phi_{0}(x), \text{\,\,\,for\,\,}x\in\Omega_{p},
      \end{equation}
      and boundary conditions
      \begin{equation}\label{6}
      v(x,t)=0, \text{\,\,\,on\,\,}(\Gamma_{f}\setminus I)\times[0,\text{\,}T] \text{\,\,\,\,\,and\,\,\,\,\,} \phi(x,t)=0, \text{\,\,\,on\,\,}(\Gamma_{p}\setminus I)\times[0,\text{\,}T],
     \end{equation}
     \begin{equation*}
                   + \text{\,\,coupling\,\,\,conditions\,\,\,across\,\,}I,
     \end{equation*}
    where $T>0$ is a constant. The interface conditions are the conservation of mass, balance of normal forces and Beavers-Joseph-Saffman condition on the tangential vectors given in \cite{30jkmt,32jkmt} by
    \begin{equation}\label{7}
        v^{T}n_{f}=\frac{1}{\eta}\mathcal{K}(\nabla\phi)^{T}n_{p}, \text{\,\,\,\,\,\,\,\,on\,\,}I,
    \end{equation}
    \begin{equation}\label{8}
        p-\nu n_{f}^{T}(\nabla v)n_{f}=\rho g\phi, \text{\,\,\,\,\,\,\,\,on\,\,}I,
    \end{equation}
    \begin{equation}\label{10}
     \nu(\tau_{l})^{T}\frac{\partial v}{\partial n_{f}}=-\frac{\alpha}{\sqrt{\tau_{l}^{T}\mathcal{K}\tau_{l}}}v^{T}\tau_{l},\text{\,\,\,\,\,\,\,\,on\,\,}I,
    \end{equation}
     for $l=1,...,m-1$, where $u^{T}$ designates the transpose of the vector $u$. Here: $\nu$, $v$, $p$ and $f$ denote the kinematic viscosity, Stokes velocity of the fluid flow in $\Omega_{f}$, kinetic pressure and external force, respectively, while $\eta$, $S_{0}$, $\phi$, $\mathcal{K}$, $g_{0}$ and $v_{p}$ represent the volumetric porosity, specific mass storativity, piezometric head, hydraulic conductivity tensor, source term and fluid velocity in $\Omega_{p}$, respectively. Furthermore, $v_{t}$ designates $\frac{\partial v}{\partial t}$, $\Delta$ and $\nabla$ mean the laplacian and gradient operators, respectively. $\rho$ denotes the fluid density whereas $g$ represents the gravitational acceleration. We assume that $\mathcal{K}\in\mathcal{M}_{m} (\mathcal{C}(\overline{\Omega}_{p}))$ (space of $m\times m$ matrices with elements in $\mathcal{C}(\overline{\Omega}_{p})$ which is the space of continuous functions over the compact domain $\overline{\Omega}_{p}$), is a non diagonal symmetric positive definite matrix. This indicates that the porous media is nonhomogeneous.\\

      In this paper, we construct a hybrid higher-order unconditionally stable approach for solving the evolutionary coupled groundwater-surface water problem $(\ref{1})$-$(\ref{6})$. As discussed in \cite{en,12en1}, there are many reasons that have motivated several researchers to develop efficient approaches for decoupling multiphysic coupled problems and thus derive single model solvers that could be applied with little extra computational and software overhead. However, some authors \cite{13en1,12en1,3en1} have constructed decoupled schemes for the stationary Stokes-Darcy model whereas the evolutionary groundwater-surface water problem has been deeply analyzed by the use of iterative and non-iterative solvers \cite{15en1,en3,14en1,jkmt,en4}. The numerical method we propose in a computed solution of the initial-boundary value problem $(\ref{1})$-$(\ref{6})$ is a three-level scheme motivated by the form of coupling. The time derivative is approximated by a second order interpolation technique while we approximate the space derivatives using the finite element method. The obtained computational scheme, so called a hybrid higher-order interpolation/finite element approach is unconditionally stable without the addition of stabilization terms (see for example the discussion in \cite{jkmt} regarding the addition of stabilization terms), temporal second-order accurate and spatial convergent with order $O(h^{d+1})$, where $h$ denotes the mesh space and $d$ is a positive integer. Additionally, because of its unconditional stability and higher-order accuracy it's a satisfactory computational technique for solving high Reynolds number flow where the elliptic interaction region can be compared to subsonic region thickness. This should lead to small time steps and thus higher computational costs if explicit numerical schemes, hybrid methods without stabilization terms or first-order implicit procedures are used such as: MacCormack approach \cite{1en1}, Leap-frog/Crank-Nicolson formulation \cite{5jkmt}, decoupled/decoupling techniques \cite{5en1,26jkmt} and MacCormack rapid solver \cite{en5}. Furthermore, the developed higher-order numerical approach is easy to implement, faster and more efficient than a large class of numerical schemes developed in the literature \cite{11zd,5en1,en3,qhpl,11zd} for solving the initial-boundary value problem $(\ref{1})$-$(\ref{6})$. In addition, the new algorithm overcomes the drawbacks (severe stability restriction, lower accuracy, no stabilization term, etc...) raised by the methods mentioned above. The highlights of this work are the following.
      \begin{description}
        \item[(i)] Construction of a hybrid higher-order unconditionally stable interpolation/finite element technique in an approximate solution of the evolutionary groundwater-surface water problem $(\ref{1})$-$(\ref{6})$.
        \item[(ii)] Analysis of unconditional stability and error estimates of the proposed approach.
        \item[(iii)] Numerical experiments to confirm the theoretical studies and to demonstrate the efficiency and validity of the new algorithm.
      \end{description}

     In the remainder of the paper we proceed as follow: Section $\ref{sec2}$ deals with a detailed description of a hybrid higher-order interpolation/finite element formulation for solving the initial-boundary value problem $(\ref{1})$-$(\ref{6})$. In Section $\ref{sec3}$, we analyze the unconditional stability and error estimates of the constructed computational technique while some numerical examples are carried out in Section $\ref{sec4}$. Section $\ref{sec5}$ presents the general conclusion together with our future works.

     \section{Construction of the higher-order computational approach}\label{sec2}

     \text{\,\,\,\,\,\,\,\,\,\,} This section develops a hybrid higher-order interpolation technique combined with finite element formulation for solving the time-dependent mixed Stokes-Darcy equations $(\ref{1})$-$(\ref{4})$ subject to initial-boundary conditions $(\ref{5})$-$(\ref{6})$ along with the coupling conditions across interface $I$. We establish the variational formulation of the initial-boundary value problem $(\ref{1})$-$(\ref{6})$ and we describe the discrete weak formulation of the new hybrid higher-order approach.\\

       We consider the Sobolev spaces
       \begin{equation*}
        V_{f}=\{u\in\left[H^{1}(\Omega_{f})\right]^{m}:\text{\,\,\,}u=0\text{\,\,on\,\,}\Gamma_{f}\setminus I\},
        \text{\,\,\,\,\,\,\,\,\,\,\,\,\,\,}V_{p}=\{\psi\in H^{1}(\Omega_{p}):\text{\,\,\,}\psi=0\text{\,\,on\,\,}\Gamma_{p}\setminus I\},
       \end{equation*}
        \begin{equation*}
       V=V_{f}\times V_{p} \text{\,\,\,\,\,\,and\,\,\,\,\,\,\,} Q=L_{0}^{2}(\Omega_{f})=\{q\in L^{2}(\Omega_{f}),\text{\,\,}\int_{\Omega}qdx=0\}.
       \end{equation*}

       The spaces $V_{s}$, for $s\in\{f,\text{\,}p\}$, are endowed with the $L^{2}$-scalar products $\left(\cdot,\cdot\right)_{0}$ and $\left(\cdot,\cdot\right)_{s}$, and the norms $\|\cdot\|_{0}$ and $\|\cdot\|_{s}$, while the space $L^{2}(I)$ is equipped with the $L^{2}$-inner product $\left(\cdot,\cdot\right)_{I}$ and $L^{2}$-norm $\|\cdot\|_{I}$. These scalar-products and norms are defined as:
       \begin{equation*}
       \left(u,v\right)_{0}=\int_{\Omega_{f}}u^{T}v dx,\text{\,\,\,}\left(u,v\right)_{f}=\int_{\Omega_{f}}\nabla u:\nabla vdx,
       \text{\,\,\,}\|u\|_{0}=\sqrt{\left(u,u\right)_{0}},\text{\,\,\,}\|u\|_{f}=\sqrt{\left(u,u\right)_{f}},\text{\,\,\,}\forall u,v\in V_{f},
       \end{equation*}
       \begin{equation*}
       \left(\psi,\phi\right)_{0}=\int_{\Omega_{p}}\psi\phi dx,\text{\,\,\,}\left(\psi,\phi\right)_{p}=\int_{\Omega_{p}}(\nabla \psi)^{T}\nabla \phi dx, \text{\,\,\,} \|\psi\|_{0}=\sqrt{\left(\psi,\psi\right)_{0}},\text{\,\,\,}\|\psi\|_{p}=\sqrt{\left(\psi,\psi\right)_{p}},\text{\,\,\,}\forall \psi,\phi\in V_{p},
       \end{equation*}
       \begin{equation}\label{11}
       \left(u_{1},u_{2}\right)_{I}=\int_{I}(u_{1}n_{f})^{T}(u_{2}n_{f})dI,\text{\,\,\,}\|u_{1}\|_{I}=\sqrt{\left(u_{1},u_{1}\right)_{I}},\text{\,\,\,}\forall u_{1},u_{2}\in L^{2}(I),
       \end{equation}
        where the symbol ":" means the scalar product on the space of square matrices $\mathcal{M}_{m}(L^{2}(\Omega_{f}))$ defined as, for $M_{1}=(m_{1})_{ij},M_{2}=(m_{2})_{ij}\in\mathcal{M}_{m}(L^{2}(\Omega_{f}))$
       \begin{equation}\label{14a}
       M_{1}:M_{2}=\underset{i=1}{\overset{m}\sum}\underset{j=1}{\overset{m}\sum}\int_{\Omega_{f}}(m_{1})_{ij}(m_{2})_{ij}dx.
       \end{equation}

       In addition, the Sobolev space $V=V_{f}\times V_{p}$, is endowed with the scalar-products $\left(\cdot,\cdot\right)_{\bar{0}}$ and $\left(\cdot,\cdot\right)_{\nabla}$, defined as
       \begin{equation}\label{12}
       \left((u,\phi),(v,\psi)\right)_{\bar{0}}=\eta\left(u,v\right)_{0}+\rho gS_{0}\left(\phi,\psi\right)_{0},\text{\,\,}\left((u,\phi),(v,\psi)\right)_{\nabla}=
       \eta\nu\left(u,v\right)_{f}+\rho g k_{\max}\left(\phi,\psi\right)_{p},\text{\,\,}\forall (u,\phi),(v,\psi)\in V,
       \end{equation}
       and the corresponding norms $\|\cdot\|_{\bar{0}}$ and $\|\cdot\|_{\nabla}$, are given by
       \begin{equation}\label{13}
       \|(u,\phi)\|_{\bar{0}}=\sqrt{\eta\|u\|_{0}^{2}+\rho gS_{0}\|\phi\|_{0}^{2}},\text{\,\,\,}\|(u,\phi)\|_{\nabla}=\sqrt{\eta\nu\|u\|_{f}^{2}+\rho gk_{\max}\|\phi\|_{p}^{2}},\text{\,\,\,}\forall (u,\phi)\in V.
       \end{equation}

        Here, $0<k_{\max}=\max\left\{\underset{x\in\overline{\Omega}_{p}}{\max}\lambda_{l}(x),\text{\,\,for\,\,}l=1,2,...,m\right\}<\infty$, where $\lambda_{l}$ are the eigenvalues of the hydraulic tensor $\mathcal{K}$ given in Section $\ref{sec1}$. Indeed, $\mathcal{K}\in M_{m}(\mathcal{C}(\overline{\Omega}_{p}))$ is a symmetric positive definite matrix, so $\mathcal{K}$ is diagonalizable and its eigenvalues $\lambda_{l}\in\mathcal{C}(\overline{\Omega}_{p})$ satisfy $\lambda_{l}(x)>0$, for every $x\in\overline{\Omega}_{p}$. Since the formulas can become quite heavy, for the convenience of writing, we assume that the parameters: $\nu$, $\eta,$ $\rho,$ $S_{0}$ and $\alpha$ are constants. Additionally, to ensure the existence and uniqueness of a smooth solution to the evolutionary groundwater-surface water problem $(\ref{1})$-$(\ref{4})$ and for the sake of discretization, we suppose that the external force, source term and initial conditions: $f$, $g_{0}$, $v_{0}$ and $\phi_{0}$, respectively, are regular enough. The following formulas obtained by integration by parts will play an important role in our analysis
       \begin{equation}\label{14}
       \int_{\Omega_{f}}u^{T}\Delta vdx=-\int_{\Omega_{f}}\nabla u:\nabla vdx+\int_{\Gamma_{f}}u^{T}(\nabla v)n_{f}d\Gamma_{f},\text{\,\,\,}
       \int_{\Omega_{p}}\psi\nabla\cdot\nabla\phi dx=\int_{\Gamma_{p}}\psi(\nabla \phi)^{T}n_{p}d\Gamma_{p}-\int_{\Omega_{p}}(\nabla\psi)^{T}\nabla\phi dx,
       \end{equation}
       for every $u\in[H^{1}(\Omega_{f})]^{m}$, $v\in[H^{2}(\Omega_{f})]^{m}$, $\psi\in H^{1}(\Omega_{p})$ and $\phi\in H^{2}(\Omega_{p})$, where $"\nabla\cdot"$ is the divergence operator,
       $n_{s}$, for $s=f,p$, are the unit outward normal vectors on $\Gamma_{s}$. Now, multiplying both equations $(\ref{1})$ and $(\ref{3})$ by $u\in V_{f}$ and $\psi\in V_{p}$, respectively, integrating over the domains $\Omega_{f}$ and $\Omega_{p}$, we obtain
       \begin{equation*}
       \int_{\Omega_{f}}u^{T}\partial_{t} vdx-\int_{\Omega_{f}}\nu u^{T}\Delta vdx+\int_{\Omega_{f}}u^{T}\nabla p dx=\int_{\Omega_{f}}u^{T}fdx,\text{\,\,}
       \int_{\Omega_{p}}S_{0}\psi\partial_{t}\phi dx-\int_{\Omega_{p}}\nabla\cdot(\mathcal{K}\nabla\phi)\psi dx=\int_{\Omega_{p}}g_{0}\psi dx.
       \end{equation*}

       Utilizing the integration by parts $(\ref{14})$, these equations are equivalent to
       \begin{equation}\label{15}
       \int_{\Omega_{f}}u^{T}\partial_{t} vdx-\int_{\Gamma_{f}}\nu u^{T}(\nabla v)n_{f}d\Gamma_{f}+\int_{\Omega_{f}}\nu\nabla u:\nabla vdx+\int_{\Gamma_{f}}pu^{T}n_{f}d\Gamma_{f}-\int_{\Omega_{f}}p\nabla\cdot udx=\int_{\Omega_{f}}u^{T}fdx,
       \end{equation}
       \begin{equation}\label{16}
       \int_{\Omega_{p}}S_{0}\psi\partial_{t}\phi dx-\int_{\Gamma_{p}}\psi n_{p}^{T}\mathcal{K}\nabla\phi d\Gamma_{p}+\int_{\Omega_{p}}(\nabla\psi)^{T}\mathcal{K}\nabla\phi dx=\int_{\Omega_{p}}g_{0}\psi dx.
       \end{equation}

       Equation $(\ref{15})$ can be rewritten as
       \begin{equation}\label{17}
       \int_{\Omega_{f}}u^{T}\partial_{t} vdx+\int_{\Omega_{f}}\nu\nabla u:\nabla vdx-\int_{\Omega_{f}}p\nabla\cdot udx+\int_{\Gamma_{f}}(pu^{T}n_{f}-
       \nu u^{T}(\nabla v)n_{f})d\Gamma_{f}=\int_{\Omega_{f}}u^{T}fdx.
       \end{equation}

       Since  $u=0$ on $\Gamma_{f}\setminus I$, and the set $\{n_{f},\tau_{l},\text{\,\,for\,\,}l=1,2,...,m-1\}$ forms a basis of $\mathbb{R}^{m}$, so $u$ can be expressed in the interface $I$ as $u=(u^{T}n_{f})n_{f}+\underset{l=1}{\overset{m-1}\sum}(u^{T}\tau_{l})\tau_{l}$. Additionally, $(\nabla v)n_{f}=\frac{\partial v}{\partial n_{f}}$, thus direct calculations give
       \begin{equation*}
       \int_{\Gamma_{f}}\nu u^{T}(\nabla v)n_{f}d\Gamma_{f}=\int_{I}\nu\left[(u^{T}n_{f})n_{f}^{T}(\nabla v)n_{f}+
       \underset{l=1}{\overset{m-1}\sum}(u^{T}\tau_{l})\tau_{l}^{T}\frac{\partial v}{\partial n_{f}}\right]dI.
       \end{equation*}

       Substituting this into equation $(\ref{17})$ and using equations $(\ref{8})$ and $(\ref{10})$, it is not hard to observe that
       \begin{equation*}
       \int_{\Omega_{f}}u^{T}\partial_{t} vdx+\int_{\Omega_{f}}\nu\nabla u:\nabla vdx-\int_{\Omega_{f}}p\nabla\cdot udx+\int_{I}u^{T}n_{f}(p-
       \nu n_{f}^{T}(\nabla v)n_{f})dI+
       \end{equation*}
       \begin{equation*}
       \underset{l=1}{\overset{m-1}\sum}\int_{I}\frac{\alpha}{\sqrt{\tau_{l}^{T}\mathcal{K}\tau_{l}}}(u^{T}\tau_{l})(v^{T}\tau_{l})dI=\int_{\Omega_{f}}u^{T}fdx,
       \end{equation*}
       which is equivalent to
       \begin{equation*}
       \int_{\Omega_{f}}u^{T}\partial_{t} vdx+\int_{\Omega_{f}}\nu\nabla u:\nabla vdx-\int_{\Omega_{f}}p\nabla\cdot udx+
       \underset{l=1}{\overset{m-1}\sum}\int_{I}\frac{\alpha}{\sqrt{\tau_{l}^{T}\mathcal{K}\tau_{l}}}(u^{T}\tau_{l})(v^{T}\tau_{l})dI
       +\int_{I}\rho g\phi u^{T}n_{f}dI=\int_{\Omega_{f}}u^{T}fdx.
       \end{equation*}

       Multiplying this equation by $\eta$, equation $(\ref{16})$ by $\rho g$, using equation $(\ref{7})$ along with the fact that $u=0$ on $\Gamma_{p}\setminus I$, and adding the obtained equations to get
       \begin{equation*}
       \eta\int_{\Omega_{f}}u^{T}\partial_{t} vdx+\rho g\int_{\Omega_{p}}S_{0}\psi\partial_{t}\phi dx+\eta\int_{\Omega_{f}}\nu\nabla u:\nabla vdx+
       \eta\underset{l=1}{\overset{m-1}\sum}\int_{I}\frac{\alpha}{\sqrt{\tau_{l}^{T}\mathcal{K}\tau_{l}}}(u^{T}\tau_{l})(v^{T}\tau_{l})dI+
       \end{equation*}
       \begin{equation}\label{18}
       \rho g\int_{\Omega_{p}}(\nabla\psi)^{T}\mathcal{K}\nabla\phi dx -\eta\int_{\Omega_{f}}p\nabla\cdot udx+\eta\rho g\int_{I}(\phi u^{T}n_{f}-\psi v^{T}n_{f})dI=\eta\int_{\Omega_{f}}u^{T}fdx+\rho g\int_{\Omega_{p}}g_{0}\psi dx.
       \end{equation}

        Using this equation, we consider the bilinear and linear operators: $B(\cdot,\cdot)$, $b(\cdot,\cdot)$, $b_{I}(\cdot,\cdot)$ and $F(\cdot)$ defined as: for $w=(v,\phi),$ $z=(u,\psi)\in V=V_{f}\times V_{p}$, and $q\in L_{0}^{2}(\Omega_{f})$,
       \begin{equation*}
       B(w,z)=\eta\int_{\Omega_{f}}\nu\nabla u:\nabla vdx+\eta\underset{l=1}{\overset{m-1}\sum}\int_{I}\frac{\alpha}{\sqrt{\tau_{l}^{T}\mathcal{K}\tau_{l}}}
       (u^{T}\tau_{l})(v^{T}\tau_{l})dI+\rho g\int_{\Omega_{p}}(\nabla\psi)^{T}\mathcal{K}\nabla\phi dx,
       \end{equation*}
       \begin{equation}\label{19}
       b(z,q)=-\eta\int_{\Omega_{f}}q\nabla\cdot udx,\text{\,\,}b_{I}(w,z)=\eta\rho g\int_{I}(\phi u^{T}n_{f}-\psi v^{T}n_{f})dI,\text{\,\,\,}
       F(z)=\eta\int_{\Omega_{f}}u^{T}fdx+\rho g\int_{\Omega_{p}}g_{0}\psi dx.
       \end{equation}

       Plugging equations $(\ref{18})$ and $(\ref{19})$ and utilizing the scalar product $\left(\cdot,\cdot\right)_{\bar{0}}$, defined in relation $(\ref{12})$, we obtain the variational formulation of the coupled evolutionary groundwater-surface water equations $(\ref{1})$-$(\ref{4})$ with initial-boundary conditions $(\ref{5})$-$(\ref{6})$. That is, find  $w=(v,\phi),$ $z=(u,\psi)\in V=V_{f}\times V_{p}$, and $p\in L_{0}^{2}(\Omega_{f})$, so that
       \begin{equation}\label{20}
       \left(w_{t},z\right)_{\bar{0}}+B(w,z)+b(z,p)+b_{I}(w,z)=F(z),\text{\,\,\,\,\,\,}\forall z\in V,
       \end{equation}
       \begin{equation}\label{21}
       b(w,q)=0,\text{\,\,\,\,\,\,}\forall q\in L_{0}^{2}(\Omega_{f}),
       \end{equation}
       \begin{equation}\label{22}
       w(x,0)=w_{0}(x),\text{\,\,\,\,\,\,on\,\,\,}\overline{\Omega}.
       \end{equation}

       Now, let $N$ be a positive integer and $\mathcal{T}_{\sigma}=\{t_{n}=n\sigma,\text{\,}n=0,1,...,N\}$, be a uniform partition of the time interval $[0,\text{\,}T]$, where $\sigma=\frac{T}{N}$ denotes the time step. Let $\mathcal{F}_{h}$ be the finite element method (FEM) triangulation of the domain $\overline{\Omega}=\Omega_{f}\cup\Gamma_{f}\cup\Omega_{p}\cup\Gamma_{p}$, which consists of triangles or tetrahedra $E$, with maximum diameter denoted by $"h"$. Moreover, $h$ represents the mesh grid of the computational domain $\overline{\Omega}$. Additionally, we assume that $\mathcal{F}_{h}$ satisfies the following assumptions: (a) $\mathcal{F}_{h}$ is regular, the triangulations $\mathcal{F}_{fh}$ and $\mathcal{F}_{ph}$ induced on the subdomains $\overline{\Omega}_{f}$ and $\overline{\Omega}_{p}$, share the same faces/edges on the interface $I=\Gamma_{f}\cap\Gamma_{p}$ and the triangulation $\mathcal{F}_{Ih}$ induced on the interface $I$ is quasi-uniform; (b) the interior of any triangle/tetrahedron is nonempty; (c) the intersection of the interior of two different triangles/tetrahedra is empty whereas the intersection of two elements in $\mathcal{F}_{h}$ is either a common face/edge or the empty set. Under these assumptions, let $V_{h}=V_{f}^{h}\times V_{p}^{h}$ and $Q_{h}\subset L_{0}^{2}(\Omega_{f})$, be the finite element spaces approximating the solution of the unsteady mixed Stokes-Darcy problem $(\ref{1})$-$(\ref{4})$ with initial condition $(\ref{5})$ and boundary condition $(\ref{6})$. Specifically, the spaces $V_{f}^{h}$, $V_{p}^{h}$, $V_{h}$, and $Q_{h}$ are defined as:
       \begin{equation*}
       V_{f}^{h}=\{u_{h}(t)\in V_{f},\text{\,}u_{h}(t)|_{E}\in[\mathcal{P}_{d+1}(E)]^{m},\text{\,}\forall E\in\mathcal{F}_{h},\forall t\in[0,\text{\,}T]\},
       \end{equation*}
       \begin{equation*}
       V_{p}^{h}=\{\psi_{h}(t)\in V_{p},\text{\,}\psi_{h}(t)|_{E}\in\mathcal{P}_{d+1}(E),\text{\,}\forall E\in\mathcal{F}_{h},\forall t\in[0,\text{\,}T]\},
       \end{equation*}
       \begin{equation*}
       V_{h}=\{w_{h}(t)=(u_{h}(t),\psi_{h}(t))\in V,\text{\,}w_{h}(t)|_{E}\in[\mathcal{P}_{d+1}(E)]^{m}\times\mathcal{P}_{d+1}(E),\text{\,}\forall E\in\mathcal{F}_{h},\forall t\in[0,\text{\,}T]\},
       \end{equation*}
       \begin{equation}\label{28}
       Q_{h}=\{q_{h}(t)\in L_{0}^{2}(\Omega_{f}),\text{\,}q_{h}(t)|_{E}\in\mathcal{P}_{d}(E),\text{\,}\forall E\in\mathcal{F}_{h},\forall t\in[0,\text{\,}T]\},
       \end{equation}
       where $\mathcal{P}_{l}(E)$ denotes the set of all polynomials defined on $E$ with degree less than or equal $l$. We introduce the set of discretely divergence velocities defined as
       \begin{equation}\label{29}
       V_{h}^{0}=\{w_{h}(t)\in V_{h},\text{\,}b(w_{h}(t),q)=0,\text{\,\,\,}\forall q\in Q_{h},\forall t\in[0,\text{\,}T]\}.
       \end{equation}

       The Stokes finite element spaces $V_{f}^{h}$ and $Q_{h}$ defined in equations $(\ref{28})$ are assumed to satisfy the usual inf-sup condition given in \cite{11gr}, that is, there is $\beta>0$ so that
       \begin{equation}\label{29a}
       \beta\leq\underset{q_{h}\neq0}{\underset{q_{h}\in Q_{h}}\inf}\underset{u_{h}\neq0}{\underset{u_{h}\in V_{f}^{h}}\sup}\frac{\int_{\Omega_{f}}q_{h}\nabla\cdot u_{h}dx}{\|u_{h}\|_{f}\|q_{h}\|_{0}}.
       \end{equation}

       The construction of the higher-order interpolation procedure combined with the finite element approach requires the approximation of the term $w_{t}^{n+1}$ at the discrete points $(t_{n-1},w^{n-1})$, $(t_{n},w^{n})$ and $(t_{n+1},w^{n+1})$. The interpolation of $w(t)$ at these points gives
       \begin{equation*}
      w(t)=\frac{(t-t_{n})(t-t_{n-1})}{(t_{n+1}-t_{n})(t_{n+1}-t_{n-1})}w^{n+1}+\frac{(t-t_{n-1})(t-t_{n+1})}{(t_{n}-t_{n-1})(t_{n}-t_{n+1})}w^{n}+
      \frac{(t-t_{n})(t-t_{n+1})}{(t_{n-1}-t_{n})(t_{n-1}-t_{n-1})}w^{n-1}+
       \end{equation*}
       \begin{equation*}
      \frac{1}{6}(t-t_{n-1})(t-t_{n})(t-t_{n+1})w_{3t}(\epsilon (t))=\frac{1}{2\sigma^{2}}[(t-t_{n})(t-t_{n-1})w^{n+1}-2(t-t_{n-1})(t-t_{n+1})w^{n}+
       \end{equation*}
       \begin{equation}\label{23}
       (t-t_{n})(t-t_{n+1})w^{n-1}]+\frac{1}{6}(t-t_{n-1})(t-t_{n})(t-t_{n+1})w_{3t}(\epsilon (t)),
       \end{equation}
        where $w_{3t}$ denotes $\frac{\partial^{3}w}{\partial t^{3}}$ and $\epsilon(t)\in(\min\{t_{n-1},t_{n},t_{n+1},t\},\text{\,}\max\{t_{n-1},t_{n},t_{n+1},t\})$. Differentiating equation $(\ref{23})$ with respect to time yields
         \begin{equation*}
      w_{t}(t)=\frac{1}{2\sigma^{2}}[(2t-t_{n}-t_{n-1})w^{n+1}-2(2t-t_{n-1}-t_{n+1})w^{n}+(2t-t_{n}-t_{n+1})w^{n-1}]+
       \end{equation*}
       \begin{equation*}
      \frac{1}{6}\{w_{3t}(\epsilon (t))\frac{d}{dt}[(t-t_{n-1})(t-t_{n})(t-t_{n+1})]+(t-t_{n-1})(t-t_{n})(t-t_{n+1})\frac{d}{dt}(w_{3t}(\epsilon (t)))\}.
       \end{equation*}

       Hence, simple calculations provide
        \begin{equation}\label{24}
       w_{t}^{n+1}=\frac{1}{2\sigma}(3w^{n+1}-4w^{n}+w^{n-1})+\frac{\sigma^{2}}{3}w_{3t}(\epsilon(t_{n+1})).
       \end{equation}

       Applying the variational formulation $(\ref{20})$-$(\ref{22})$ at the discrete time $t_{n+1}$, utilizing approximation $(\ref{24})$ and rearranging terms to obtain
       \begin{equation}\label{25}
       \left(3w^{n+1}-4w^{n}+w^{n-1},z\right)_{\bar{0}}+2\sigma[B(w^{n+1},z)+b(z,p^{n+1})+b_{I}(w^{n+1},z)]=2\sigma F^{n+1}(z)-\frac{2\sigma^{3}}{3}\left(w_{3t}
       (\epsilon(t_{n+1})),z\right)_{\bar{0}},
       \end{equation}
       for every $z\in V$,
       \begin{equation}\label{26}
       b(w^{n+1},q)=0,\text{\,\,\,\,\,\,}\forall q\in L_{0}^{2}(\Omega_{f}),
       \end{equation}
       \begin{equation}\label{27}
       w(x,0)=w_{0}(x),\text{\,\,\,\,\,\,on\,\,\,}\overline{\Omega}.
       \end{equation}

       Tracking the error term: $-\frac{2\sigma^{3}}{3}\left(w_{3t}(\epsilon(t_{n+1})),z\right)_{\bar{0}}$, in the discrete variational formulation $(\ref{25})$-$(\ref{27})$, replacing the exact solution $w^{n+1}\in V$, with the approximate one $w_{h}^{n+1}\in V_{h}^{0}$ and rearranging terms to get the desired hybrid higher-order unconditionally stable interpolation/finite element technique for solving the initial-boundary value problem $(\ref{1})$-$(\ref{6})$, that is, given $(w_{h}^{n-1},p_{h}^{n-1}),(w_{h}^{n},p_{h}^{n})\in V_{h}^{0}\times Q_{h}$, find $(w_{h}^{n+1},p_{h}^{n+1})\in V_{h}^{0}\times Q_{h}$, for $n=1,2,...,N-1$, so that
       \begin{equation}\label{30}
       \left(w_{h}^{n+1},z\right)_{\bar{0}}+\frac{2\sigma}{3}[B(w_{h}^{n+1},z)+b(z,p_{h}^{n+1})+b_{I}(w_{h}^{n+1},z)]=\frac{1}{3}\left(4w_{h}^{n}-w_{h}^{n-1},z\right)_{\bar{0}}
       +\frac{2\sigma}{3}F^{n+1}(z),\text{\,\,\,}\forall z\in V,
       \end{equation}
       \begin{equation}\label{31}
       b(w_{h}^{n+1},q)=0,\text{\,\,\,\,\,\,}\forall q\in Q_{h},
       \end{equation}
       subject to initial condition
       \begin{equation}\label{32}
       w_{h}^{0}=\mathbb{P}_{h}w_{0}\text{\,\,\,\,and\,\,\,\,}p_{h}^{0}=P_{h}p_{0},\text{\,\,\,\,\,\,on\,\,\,}\overline{\Omega},
       \end{equation}
       where $\mathbb{P}_{h}$ and $P_{h}$ represent the $L^{2}$-projection operators from $V$ onto $V_{h}$ and $L_{0}^{2}(\Omega_{f})$ onto $Q_{h}$, respectively, that satisfy for any functions $\bar{w}\in V$ and $\bar{p}\in L_{0}^{2}(\Omega_{f})$,
       \begin{equation}\label{33}
       \left(\mathbb{P}_{h}\bar{w},z\right)_{\bar{0}}=\left(\bar{w},z\right)_{\bar{0}},\text{\,\,\,}\forall z\in V_{h}\text{\,\,\,and\,\,\,} \left(P_{h}\bar{p},q\right)_{0}=\left(\bar{p},q\right)_{0},\text{\,\,\,}\forall q\in Q_{h}.
       \end{equation}

       It follows from equations $(\ref{30})$-$(\ref{32})$ that the new algorithm is a three-level implicit approach. Hence, both terms $(w_{h}^{0},p_{h}^{0})$ and $(w_{h}^{1},p_{h}^{1})$ are required to start the developed computational technique. However, $(w_{h}^{0},p_{h}^{0})$ is provided by the initial condition $(\ref{32})$ while the term $(w_{h}^{1},p_{h}^{1})=((v_{h}^{1},\phi_{h}^{1}),p_{h}^{1})$ is obtained using the second-order Taylor polynomial of $v^{1}$ and $\phi^{1}$, the bilinear operator $b(\cdot,\cdot)$ and the projection operators $\mathbb{P}_{h}$ and $P_{h}$, respectively. That is,
       \begin{equation}\label{34}
       w_{h}^{1}=\mathbb{P}_{h}\bar{w}^{1}\text{\,\,\,and\,\,\,} b(w_{h}^{1},p^{1}_{h})=0,
       \end{equation}
       where
       \begin{equation}\label{35}
       \bar{w}^{1}=(\bar{v}^{1},\bar{\phi}^{1}),\text{\,\,\,\,\,}\bar{v}^{1}=v_{0}+\nu\sigma\Delta v_{0}-\nabla p_{0}+f^{0}\text{\,\,\,and\,\,\,}\bar{\phi}^{1}=
       \phi_{0}+\frac{1}{S_{0}}\sigma\nabla\cdot(\mathcal{K}\nabla\phi_{0})+g_{0}^{0}.
       \end{equation}

        It's worth mentioning that the term $\bar{w}^{1}=(\bar{v}^{1},\bar{\phi}^{1})$ satisfies
       \begin{equation}\label{36}
       \|w^{1}-\bar{w}^{1}\|_{\bar{0}}\leq \frac{\sigma^{2}}{2}\underset{0\leq \theta\leq T}{\max}\|w_{2t}(\theta)\|_{\bar{0}}.
       \end{equation}

       In the following we will analyze the unconditional stability and the error estimates of the proposed interpolation approach combined with the finite element procedure $(\ref{30})$-$(\ref{32})$ and $(\ref{34})$ in a numerical solution of the evolutionary Stokes-Darcy problem $(\ref{1})$-$(\ref{4})$ subjects to initial condition $(\ref{5})$ and boundary condition $(\ref{6})$. We assume that the $L^{2}$-projection operators $\mathbb{P}_{h}$ and $P_{h}$ defined in equation $(\ref{33})$ satisfy the following approximation properties of piecewise polynomials
       \begin{equation}\label{37}
       \|(\mathcal{I}-\mathbb{P}_{h})w\|_{\bar{0}}\leq C_{1}h^{d+1}\|w\|_{V^{d+1}},\text{\,\,}\forall w\in V^{d+1};\text{\,\,\,}\|(\mathcal{I}_{0}-P_{h})q\|_{0}\leq C_{2}h^{d}\|q\|_{H^{d}(\Omega_{f})},\text{\,\,}\forall q\in H^{d}(\Omega_{f}),
       \end{equation}
       where $\mathcal{I}$ and $\mathcal{I}_{0}$ are identity operators, $V^{d+1}$ denotes a subspace of $[H^{d+1}(\Omega_{f})]^{m}\times H^{d+1}(\Omega_{p})$, and $C_{l}$, for $l=1,2$, are two positive constants independent of the mesh grid $h$ and time step $\sigma$. To establish the unconditional stability and error estimates of the new algorithm, we will use the following Poincar\'{e}-Friedrich and trace inequalities $(\ref{38})$ and $(\ref{39})$, respectively,
       \begin{equation}\label{38}
        \|X\|_{0}\leq \sqrt{\widetilde{C}}\|X\|_{s}, \text{\,\,\,}\forall X\in V_{s},
       \end{equation}
       \begin{equation}\label{39}
        \|X\|_{\Gamma_{s}}\leq\sqrt{\widehat{C}}\|X\|_{0}^{\frac{1}{2}}\|X\|_{s}^{\frac{1}{2}}, \text{\,\,\,}\forall X\in V_{s},
       \end{equation}
        where $s\in\{f,p\}$ and $\widetilde{C}$, $\widehat{C}>0$, are two constants which do not depend on the time step $\sigma$ and space step $h$, together with Lemmas $\ref{l1}$-$\ref{l2}$.

        \begin{lemma}\label{l1}
        For any $w=(v,\phi),\text{\,\,}z=(u,\psi)\in V$, the bilinear operator $B(\cdot,\cdot)$ is symmetric and satisfies
       \begin{equation*}
        B(w,z)\leq \left(2+\frac{\alpha(m-1)\widetilde{C}\widehat{C}}{\nu\sqrt{k_{\min}}}\right)\|w\|_{\nabla}\|z\|_{\nabla}\text{\,\,\,\,and\,\,\,\,}
        B(z,z)\geq \frac{k_{\min}}{k_{\max}}\|z\|_{\nabla}^{2},
       \end{equation*}
        where $m=2$ or $3$, $\nu>0$ is the physical parameter given in equation $(\ref{1})$, $\widetilde{C}$ and $\widehat{C}$ are constants given in estimates $(\ref{38})$ and $(\ref{39})$, respectively,
        \begin{equation*}\label{}
        k_{\min}=\min\left\{\underset{x\in\overline{\Omega}_{p}}{\min}\lambda_{l}(x),\text{\,\,for\,\,}l=1,2,...,m\right\}>0\text{\,\,\,and\,\,\,}
        k_{\max}=\max\left\{\underset{x\in\overline{\Omega}_{p}}{\max}\lambda_{l}(x),\text{\,\,for\,\,}l=1,2,...,m\right\}<\infty,
       \end{equation*}
         where $\lambda_{l}$ are the eigenvalues of the hydraulic tensor $\mathcal{K}$ defined in Section $\ref{sec1}$. In fact, $\mathcal{K}\in M_{m}(\mathcal{C}(\overline{\Omega}_{p}))$ is a symmetric positive definite tensor, then $\mathcal{K}$ is diagonalizable and its eigenvalues $\lambda_{l}\in\mathcal{C}(\overline{\Omega}_{p})$ satisfy $\lambda_{l}(x)>0$, for every $x\in\overline{\Omega}_{p}$ (compact subset of $\mathbb{R}^{m}$), and $\|\cdot\|_{\nabla}$ denotes the norm defined in relation $(\ref{13})$.
       \end{lemma}

       \begin{proof}
       For every $w=(v,\phi),$ $z=(u,\psi)\in V$, utilizing equation $(\ref{19})$, it is not hard to observe that $B(w,z)=B(z,w)$, so the bilinear operator $B(\cdot,\cdot)$ is symmetric. Furthermore, it follows from equation $(\ref{19})$ along with the scalar products and norms given in relation $(\ref{11})$ that
       \begin{equation*}
       B(w,z)=\eta\int_{\Omega_{f}}\nu\nabla u:\nabla vdx+\rho g\int_{\Omega_{p}}(\nabla\psi)^{T}\mathcal{K}\nabla\phi dx+ \eta\underset{l=1}{\overset{m-1}\sum}\int_{I}\frac{\alpha}{\sqrt{\tau_{l}^{T}\mathcal{K}\tau_{l}}}(u^{T}\tau_{l})(v^{T}\tau_{l})dI\leq \eta\nu\left|\int_{\Omega_{f}}\nabla u:\nabla vdx\right|
       \end{equation*}
       \begin{equation*}
       +\rho g\left|\int_{\Omega_{p}}(\nabla\psi)^{T}\mathcal{K}\nabla\phi dx\right|+ \eta\alpha\underset{l=1}{\overset{m-1}\sum}
       \left|\int_{I}\frac{1}{\sqrt{\tau_{l}^{T}\mathcal{K}\tau_{l}}}(u^{T}\tau_{l})(v^{T}\tau_{l})dI\right|\leq \eta\nu\|u\|_{f}\|v\|_{f}+\rho g k_{\max}\|\phi\|_{p}\|\psi_{p}+
       \end{equation*}
       \begin{equation}\label{40}
       \frac{\eta\alpha}{\sqrt{k_{\min}}}\underset{l=1}{\overset{m-1}\sum}\left|\int_{I}(u^{T}\tau_{l})(v^{T}\tau_{l})dI\right|.
       \end{equation}

       Indeed, applying the Cauchy-Schwarz inequality, direct computations give: $\left|\int_{\Omega_{f}}\nabla u:\nabla vdx\right|=|\left(u,v\right)_{f}|\leq\|u\|_{f}\|v\|_{f}$ and $\left|\int_{\Omega_{p}}(\nabla\psi)^{T}\mathcal{K}\nabla\phi dx\right|\leq\int_{\Omega_{p}}\left|(\nabla\psi)^{T}\mathcal{K}\nabla\phi\right|dx\leq k_{\max}\int_{\Omega_{p}}\left|(\nabla\psi)^{T}\nabla\phi\right|dx\leq k_{\max}\|\phi\|_{p}\|\psi\|_{p}$. Additionally, since $\tau_{l}$, for $l=1,...,m-1$, are unit tangential vector on the coupling interface $I$, it holds: $\tau_{l}^{T}\mathcal{K}\tau_{l}\geq k_{\min}\tau_{l}^{T}\tau_{l}=k_{\min}$. Now, using the Cauchy-Schwarz inequality along with the norm $\|\cdot\|_{I}$ defined in relation $(\ref{11})$, it is not hard to see that
       \begin{equation*}
       \left|\int_{I}(u^{T}\tau_{l})(v^{T}\tau_{l})dI\right|=\left|\int_{I}(u^{T}\tau_{l})(\tau_{l}^{T}v)dI\right|=\left|\int_{I}u^{T}\tau_{l}\tau_{l}^{T}vdI\right|=
       \left|\int_{I}u^{T}vdI\right|=\left|\int_{I}(n_{f}^{T}u)^{T}(n_{f}^{T}v)dI\right|\leq \|u\|_{I}\|v\|_{I},
       \end{equation*}
       since $\tau_{l}\tau_{l}^{T}$ and $n_{f}n_{f}^{T}$ are matrices of orthogonal projections. Indeed, $(\tau_{l}\tau_{l}^{T})\tau_{l}\tau_{l}^{T}=\tau_{l}\underset{=1}
       {\underbrace{(\tau_{l}^{T}\tau_{l})}}\tau_{l}^{T}=\tau_{l}\tau_{l}^{T}$ and $(n_{f}n_{f}^{T})n_{f}n_{f}^{T}=n_{f}\underset{=1}{\underbrace{(n_{f}^{T}n_{f})}}n_{f}^{T}=
       n_{f}n_{f}^{T}$. Applying both trace and Poincar\'{e}-Friedrich inequalities $(\ref{39})$ and $(\ref{38})$, and using the norm $\|\cdot\|_{\nabla}$ given in equation $(\ref{13})$, simple calculations yield
       \begin{equation*}
       \left|\int_{I}(u^{T}\tau_{l})(v^{T}\tau_{l})dI\right|\leq \|u\|_{I}\|v\|_{I}\leq \widehat{C}\|u\|_{0}^{\frac{1}{2}}\|u\|_{f}^{\frac{1}{2}}
       \|v\|_{0}^{\frac{1}{2}}\|v\|_{f}^{\frac{1}{2}}\leq\widehat{C}\widetilde{C}\|u\|_{f}\|v\|_{f}\leq
       \end{equation*}
       \begin{equation*}
       \widehat{C}\widetilde{C}(\eta\nu)^{-1}(\eta\nu\|u\|_{f}^{2}+\rho gk_{\max}\|\psi\|_{p}^{2})^{\frac{1}{2}}(\eta\nu\|v\|_{f}^{2}+\rho gk_{\max}\|\phi\|_{p}^{2})^{\frac{1}{2}}=\widehat{C}\widetilde{C}(\eta\nu)^{-1}\|w\|_{\nabla}\|z\|_{\nabla}.
       \end{equation*}

       Substituting this estimate into inequality $(\ref{40})$ and performing direct calculations to obtain
       \begin{equation*}
       B(w,z)\leq \eta\nu\|u\|_{f}\|v\|_{f}+\rho g k_{\max}\|\phi\|_{p}\|\psi\|_{p}+\frac{\alpha(m-1)\widehat{C}\widetilde{C}}{\nu\sqrt{k_{\min}}}\|w\|_{\nabla}\|z\|_{\nabla}=
       (\eta\nu\|u\|_{f}^{2})^{\frac{1}{2}}(\eta\nu\|v\|_{f}^{2})^{\frac{1}{2}}+
       \end{equation*}
       \begin{equation*}
       (\rho g k_{\max}\|\phi\|_{p}^{2})^{\frac{1}{2}}(\rho g k_{\max}\|\psi\|_{p}^{2})^{\frac{1}{2}}+\frac{\alpha(m-1)\widehat{C}\widetilde{C}}{\nu\sqrt{ k_{\min}}}
       \|w\|_{\nabla}\|z\|_{\nabla}\leq 2(\eta\nu\|u\|_{f}^{2}+\rho g k_{\max}\|\psi\|_{p}^{2})^{\frac{1}{2}}\times
       \end{equation*}
       \begin{equation}\label{40a}
       (\eta\nu\|v\|_{f}^{2}+\rho g k_{\max}\|\phi\|_{p}^{2})^{\frac{1}{2}}+\frac{\alpha(m-1)\widehat{C}\widetilde{C}}{\nu\sqrt{ k_{\min}}}\|w\|_{\nabla}\|z\|_{\nabla}=
       \left(2+\frac{\alpha(m-1)\widehat{C}\widetilde{C}}{\nu\sqrt{ k_{\min}}}\right)\|w\|_{\nabla}\|z\|_{\nabla},
       \end{equation}
       where $"\times"$ means the usual multiplication in $\mathbb{R}$. This ends the proof of the first inequality in Lemma $\ref{l1}$. In addition,
       \begin{equation*}
        B(z,z)=\eta\int_{\Omega_{f}}\nu\nabla u:\nabla udx+\rho g\int_{\Omega_{p}}(\nabla\psi)^{T}\mathcal{K}\nabla\psi dx+\eta\overset{m-1}{\underset{l=1}\sum}\int_{I}\frac{\alpha}{\sqrt{\tau_{l}^{T}\mathcal{K}\tau_{l}}}(u^{T}\tau_{l})^{2} dI\geq\eta\nu\|u\|_{f}^{2}+\rho g k_{\min}\|\psi\|_{p}^{2}+
       \end{equation*}
       \begin{equation}\label{40b}
        \eta\overset{m-1}{\underset{l=1}\sum}\int_{I}\frac{\alpha}{\sqrt{\tau_{l}^{T}\mathcal{K}\tau_{l}}}(u^{T}\tau_{l})^{2} dI \geq \eta\nu\|u\|_{f}^{2}
        +\frac{k_{\min}}{k_{\max}}\rho gk_{\max}\|\psi\|_{p}^{2}\geq \frac{k_{\min}}{k_{\max}}\|z\|_{\nabla}^{2}.
       \end{equation}

       The first inequality follows from:
       \begin{equation*}
        \int_{\Omega_{p}}(\nabla\psi)^{T}\mathcal{K}\nabla\psi dx=\int_{\Omega_{p}}(\nabla\psi)^{T}\mathcal{O}diag(\lambda_{1}(x),\cdots,\lambda_{m}(x))\mathcal{O}^{T}\nabla\psi dx\geq k_{\min}\int_{\Omega_{p}}(\nabla\psi)^{T}\nabla\psi dx = k_{\min}\|\psi\|_{p}^{2},
       \end{equation*}
        where the tensor $\mathcal{K}$ is decomposed as: $\mathcal{K}=\mathcal{O}diag(\lambda_{1},\cdots,\lambda_{m})\mathcal{O}^{T}$ and $\mathcal{O}$ is the orthogonal matrix formed with the eigenvectors of $\mathcal{K}$. This completes the proof of Lemma $\ref{l1}$.
       \end{proof}

       \begin{remark}
       It's worth mentioning that Lemma $\ref{l1}$ indicates that the bilinear operator $B(\cdot,\cdot)$ defines a scalar product on the Sobolev space $V\times V$. Let $\|\cdot\|_{B}$ be the corresponding norm, so for $z=(u,\psi)\in V$, $\|\cdot\|_{B}$ is defined as
       \begin{equation}\label{41}
       \|z\|_{B}^{2}=B(z,z)=\eta\nu\int_{\Omega_{f}}\nabla u:\nabla udx+\rho g\int_{\Omega_{p}}(\nabla\psi)^{T}\mathcal{K}\nabla\psi dx+\eta\overset{m-1}{\underset{l=1}\sum}\int_{I}\frac{\alpha}{\sqrt{\tau_{l}^{T}\mathcal{K}\tau_{l}}}(u^{T}\tau_{l})^{2} dI,
       \end{equation}
       where $\tau_{l}$, $l=1,2,...,m-1$, are the linearly independent unit tangential vectors on the interface $I$. Furthermore, using estimates $(\ref{40a})$ and $(\ref{40b})$, it's not difficult to observe that both norms $\|\cdot\|_{\nabla}$ and $\|\cdot\|_{B}$ are equivalent on the Sobolev space $V$, that is,
       \begin{equation}\label{41a}
        \sqrt{\frac{k_{\min}}{k_{\max}}}\|z\|_{\nabla} \leq \|z\|_{B}\leq\sqrt{\left(2+\frac{\alpha(m-1)\widehat{C}\widetilde{C}}{\nu\sqrt{ k_{\min}}}\right)}\|z\|_{\nabla}, \text{\,\,\,\,}\forall z\in V.
       \end{equation}

       Furthermore, the ratio $\frac{k_{\max}}{k_{\min}}$ should be observed as the condition number of the symmetric positive definite matrix $\mathcal{K}$.
       \end{remark}

       \begin{lemma}\label{l2}
        The bilinear operator $b_{I}(\cdot,\cdot)$ defined in relation $(\ref{19})$ is anti-symmetric.
       \end{lemma}

       \begin{proof}
        For any $w=(v,\phi)$ and $z=(u,\psi)\in V$,
       \begin{equation*}
        b_{I}(w,z)=\eta\rho g\int_{I}(\phi u^{T}n_{f}-\psi v^{T}n_{f})dI=-\eta\rho g\int_{I}(\psi v^{T}n_{f}-\phi u^{T}n_{f})dI=-b_{I}(z,w).
       \end{equation*}

       Thus $b_{I}(\cdot,\cdot)$ is anti-symmetric.
       \end{proof}

        In the following, we assume that the exact solution $(w,p)$ of the mixed Stokes-Darcy model $(\ref{1})$-$(\ref{6})$ satisfies the regularity condition $(w,p)\in H^{4}(0,T;\text{\,}V^{d+1})\times L^{2}(0,T;\text{\,}L^{2}_{0}(\Omega_{f}))$. Thus, there is a positive constant $\widehat{C}_{2}$ independent of the grid space $h$ and time step $\sigma$, so that
        \begin{equation}\label{42}
       \||w|\|_{B,\infty}+\||w_{2t}|\|_{B,\infty}+\||w_{3t}|\|_{B,\infty}=\underset{0\leq \theta\leq T}{\max}\|w(\theta)\|_{B}+\underset{0\leq\theta\leq T}{\max}\|w_{2t}(\theta)\|_{B}+\underset{0\leq \theta\leq T}{\max}\|w_{3t}(\theta)\|_{B}\leq \widehat{C}_{2}.
       \end{equation}

         \section{Stability analysis and error estimates}\label{sec3}

       \text{\,\,\,\,\,\,\,\,\,\,} In this section, we establish the unconditional stability and error estimates of the developed computational approach $(\ref{30})$-$(\ref{32})$ and $(\ref{34})$ for solving the initial-boundary value problem $(\ref{1})$-$(\ref{6})$.

       \begin{theorem}\label{t1} (Unconditional stability and error estimates).
          Let $(w,p)=((v,\phi),p)\in H^{4}(0,T;\text{\,}V^{d+1})\times L^{2}(0,T;\text{\,}L^{2}_{0}(\Omega_{f}))$ be the analytical solution of the unsteady mixed Stokes-Darcy problem $(\ref{1})$-$(\ref{4})$ subjects to initial-boundary conditions $(\ref{5})$-$(\ref{6})$ and let $(w_{h},p_{h})=((v_{h},\phi_{h}),p_{h})\in V_{h}^{0}\times Q_{h}$ be the approximate one provided by the new algorithm $(\ref{30})$-$(\ref{32})$ and $(\ref{34})$. Thus, it holds
          \begin{equation*}
        \||w_{h}|\|_{\bar{0},\infty}^{2}+2\sigma\underset{l=1}{\overset{n}\sum}\|w_{h}^{l+1}-w^{l+1}\|_{B}^{2}\leq 2\||w|\|_{B,\infty}^{2}
        +2\max\left\{\frac{\widetilde{C}_{1}k_{\max}}{k_{\min}}(\frac{5}{2}\||w_{2t}|\|_{B,\infty}^{2}+\frac{4T\widetilde{C}_{1}k_{\max}}{27k_{\min}}
        \widetilde{C}_{1}\||w_{3t}|\|_{B,\infty}^{2}),\right.
       \end{equation*}
        \begin{equation}\label{43}
        \left.C_{1}^{2}(10\|w^{1}\|_{V^{d+1}}^{2}+3\|w_{0}\|_{V^{d+1}}^{2})\right\}(\sigma^{2}+h^{d+1})^{2},
        \end{equation}
         \begin{equation*}
        \||e_{h}|\|_{\bar{0},\infty}^{2}+\sigma\underset{l=1}{\overset{n}\sum}\|e_{w,h}^{l+1}\|_{B}^{2} \leq \max\left\{\frac{\widetilde{C}_{1}k_{\max}}{k_{\min}}
        (\frac{5}{2}\||w_{2t}|\|_{B,\infty}^{2}+\frac{4T\widetilde{C}_{1}k_{\max}}{27k_{\min}}\widetilde{C}_{1}\||w_{3t}|\|_{B,\infty}^{2}),\right.
       \end{equation*}
        \begin{equation}\label{44}
        \left. C_{1}^{2}(10\|w^{1}\|_{V^{d+1}}^{2}+3\|w_{0}\|_{V^{d+1}}^{2})\right\}(\sigma^{2}+h^{d+1})^{2},
        \end{equation}
        where $\widetilde{C}_{1}=\widetilde{C}\max\{\nu^{-1},S_{0}k_{\max}^{-1}\}$ and $e_{w,h}(t)=w_{h}(t)-w(t)$, for any $t\in[0,\text{\,}T]$, is the corresponding error term.
        \end{theorem}

        \begin{proof}
        Subtracting equation $(\ref{25})$ from $(\ref{30})$, rearranging terms and replacing $w_{h}^{l}-w^{l}$ and $p_{h}^{l}-p^{l}$ with $e_{w,h}^{l}$ and $e_{p,h}^{l}$, for $l=n-1$, $n$ and $n+1$, respectively, result in
        \begin{equation*}
       \left(3e_{w,h}^{n+1}-4e_{w,h}^{n}+e_{w,h}^{n-1},z\right)_{\bar{0}}+2\sigma[B(e_{w,h}^{n+1},z)+b(z,e_{p,h}^{n+1})+b_{I}(e_{w,h}^{n+1},z)]=\frac{2\sigma^{3}}{3}\left(w_{3t}
       (\epsilon(t_{n+1})),z\right)_{\bar{0}},\text{\,\,\,}\forall z\in V.
       \end{equation*}

       For $z=2e_{w,h}^{n+1}$, this equation becomes
       \begin{equation*}
       \left(3e_{w,h}^{n+1}-4e_{w,h}^{n}+e_{w,h}^{n-1},2e_{w,h}^{n+1}\right)_{\bar{0}}+2\sigma[B(e_{w,h}^{n+1},2e_{w,h}^{n+1})+b(2e_{w,h}^{n+1},e_{p,h}^{n+1})+
       b_{I}(e_{w,h}^{n+1},2e_{w,h}^{n+1})]=
       \end{equation*}
       \begin{equation}\label{45}
       \frac{2\sigma^{3}}{3}\left(w_{3t}(\epsilon(t_{n+1})),2e_{w,h}^{n+1}\right)_{\bar{0}}.
       \end{equation}

       Since the bilinear operator $b_{I}(\cdot,\cdot)$ is anti-symmetric, it holds
       \begin{equation}\label{46}
       b_{I}(e_{w,h}^{n+1},2e_{w,h}^{n+1})=2b_{I}(e_{w,h}^{n+1},e_{w,h}^{n+1})=0.
       \end{equation}

       In addition, using equations $(\ref{26})$ and $(\ref{31})$ and the fact that $b(\cdot,\cdot)$ is a bilinear operator to get
       \begin{equation*}
       b(2e_{w,h}^{n+1},e_{p,h}^{n+1})=2b(e_{w,h}^{n+1},e_{p,h}^{n+1})=2b(w_{h}^{n+1}-w^{n+1},p_{h}^{n+1}-p^{n+1})=
       \end{equation*}
       \begin{equation}\label{47}
       2[b(w_{h}^{n+1},p_{h}^{n+1})-b(w_{h}^{n+1},p^{n+1})-b(w^{n+1},p_{h}^{n+1})+b(w^{n+1},p^{n+1})]=0,
       \end{equation}
       since $b(w_{h}^{n+1},p_{h}^{n+1})=b(w_{h}^{n+1},p^{n+1})=b(w^{n+1},p_{h}^{n+1})=b(w^{n+1},p^{n+1})=0$. Additionally, it is easy to observe that
       \begin{equation*}
       \left(3e_{w,h}^{n+1}-4e_{w,h}^{n}+e_{w,h}^{n-1},2e_{w,h}^{n+1}\right)_{\bar{0}}=\left(3e_{w,h}^{n+1}-4e_{w,h}^{n}+e_{w,h}^{n-1},e_{w,h}^{n+1}
       +e_{w,h}^{n-1}\right)_{\bar{0}}+
       \end{equation*}
       \begin{equation*}
       \left(3e_{w,h}^{n+1}-4e_{w,h}^{n}+e_{w,h}^{n-1},e_{w,h}^{n+1}-e_{w,h}^{n-1}\right)_{\bar{0}}.
       \end{equation*}

       Performing straightforward computations provide
       \begin{equation*}
       \left(3e_{w,h}^{n+1}-4e_{w,h}^{n}+e_{w,h}^{n-1},e_{w,h}^{n+1}+e_{w,h}^{n-1}\right)_{\bar{0}}=\left(3(e_{w,h}^{n+1}-e_{w,h}^{n})-(e_{w,h}^{n}-e_{w,h}^{n-1}),
       e_{w,h}^{n+1}+e_{w,h}^{n-1}\right)_{\bar{0}}=
       \end{equation*}
       \begin{equation*}
       3\left(e_{w,h}^{n+1}-e_{w,h}^{n},e_{w,h}^{n+1}\right)_{\bar{0}}+3\left(e_{w,h}^{n+1}-e_{w,h}^{n},e_{w,h}^{n-1}\right)_{\bar{0}}-
       \left(e_{w,h}^{n}-e_{w,h}^{n-1},e_{w,h}^{n+1}\right)_{\bar{0}}-\left(e_{w,h}^{n}-e_{w,h}^{n-1},e_{w,h}^{n-1}\right)_{\bar{0}}=
       \end{equation*}
       \begin{equation*}
       3\left(e_{w,h}^{n+1}-e_{w,h}^{n},e_{w,h}^{n+1}\right)_{\bar{0}}+\left(e_{w,h}^{n-1}-e_{w,h}^{n},e_{w,h}^{n-1}\right)_{\bar{0}}+
       3[\left(e_{w,h}^{n+1}-e_{w,h}^{n-1},e_{w,h}^{n-1}\right)_{\bar{0}}+\left(e_{w,h}^{n-1}-e_{w,h}^{n},e_{w,h}^{n-1}\right)_{\bar{0}}]-
       \end{equation*}
       \begin{equation}\label{48}
       [\left(e_{w,h}^{n}-e_{w,h}^{n+1},e_{w,h}^{n+1}\right)_{\bar{0}}+\left(e_{w,h}^{n+1}-e_{w,h}^{n-1},e_{w,h}^{n+1}\right)_{\bar{0}}].
       \end{equation}

       Since $\left(\cdot,\cdot\right)_{\bar{0}}$ is a scalar product on the space $V\times V\subset[(H^{1}(\Omega_{f}))^{m}\times H^{1}(\Omega_{f})]\times [(H^{1}(\Omega_{f}))^{m}\times H^{1}(\Omega_{f})]$, with corresponding norm $\|\cdot\|_{\bar{0}}$, so $\left(u_{1}\pm u_{2},u_{1}\right)_{\bar{0}}=\frac{1}{2}(\|u_{1}\pm u_{2}\|_{\bar{0}}^{2}+\|u_{1}\|_{\bar{0}}^{2}-\|u_{2}\|_{\bar{0}}^{2})$, for every vectors $u_{1},u_{2}\in V$. Utilizing this fact, equation $(\ref{48})$ becomes
       \begin{equation*}
       \left(3e_{w,h}^{n+1}-4e_{w,h}^{n}+e_{w,h}^{n-1},e_{w,h}^{n+1}+e_{w,h}^{n-1}\right)_{\bar{0}}=\frac{3}{2}\left(\|e_{w,h}^{n+1}-e_{w,h}^{n}\|_{\bar{0}}^{2}+
       \|e_{w,h}^{n+1}\|_{\bar{0}}^{2}-\|e_{w,h}^{n}\|_{\bar{0}}^{2}\right)+\frac{1}{2}\left(\|e_{w,h}^{n-1}-e_{w,h}^{n}\|_{\bar{0}}^{2}+\right.
       \end{equation*}
       \begin{equation*}
       \left.\|e_{w,h}^{n-1}\|_{\bar{0}}^{2}-\|e_{w,h}^{n}\|_{\bar{0}}^{2}\right)+\frac{3}{2}\left[-\left(\|e_{w,h}^{n-1}-e_{w,h}^{n+1}\|_{\bar{0}}^{2}+
       \|e_{w,h}^{n-1}\|_{\bar{0}}^{2}-\|e_{w,h}^{n+1}\|_{\bar{0}}^{2}\right)+\|e_{w,h}^{n-1}-e_{w,h}^{n}\|_{\bar{0}}^{2}+\|e_{w,h}^{n-1}\|_{\bar{0}}^{2}-
       \|e_{w,h}^{n}\|_{\bar{0}}^{2}\right]
       \end{equation*}
       \begin{equation*}
        +\frac{1}{2}\left[\|e_{w,h}^{n+1}-e_{w,h}^{n}\|_{\bar{0}}^{2}+\|e_{w,h}^{n+1}\|_{\bar{0}}^{2}-\|e_{w,h}^{n}\|_{\bar{0}}^{2}-
        \left(\|e_{w,h}^{n+1}-e_{w,h}^{n-1}\|_{\bar{0}}^{2}+\|e_{w,h}^{n+1}\|_{\bar{0}}^{2}-\|e_{w,h}^{n-1}\|_{\bar{0}}^{2}\right)\right]=
       \end{equation*}
       \begin{equation}\label{49}
        2(\|e_{w,h}^{n+1}-e_{w,h}^{n}\|_{\bar{0}}^{2}+\|e_{w,h}^{n-1}-e_{w,h}^{n}\|_{\bar{0}}^{2})-2\|e_{w,h}^{n+1}-e_{w,h}^{n-1}\|_{\bar{0}}^{2}+
       3\|e_{w,h}^{n+1}\|_{\bar{0}}^{2}-4\|e_{w,h}^{n}\|_{\bar{0}}^{2}+\|e_{w,h}^{n-1}\|_{\bar{0}}^{2}.
       \end{equation}

       In a similar manner, one easily shows that
       \begin{equation*}
       \left(3e_{w,h}^{n+1}-4e_{w,h}^{n}+e_{w,h}^{n-1},e_{w,h}^{n+1}-e_{w,h}^{n-1}\right)_{\bar{0}}=\frac{3}{2}\left(\|e_{w,h}^{n+1}-e_{w,h}^{n}\|_{\bar{0}}^{2}+
       \|e_{w,h}^{n+1}\|_{\bar{0}}^{2}-\|e_{w,h}^{n}\|_{\bar{0}}^{2}\right)-\frac{1}{2}\left(\|e_{w,h}^{n-1}-e_{w,h}^{n}\|_{\bar{0}}^{2}+\right.
       \end{equation*}
       \begin{equation*}
       \left.\|e_{w,h}^{n-1}\|_{\bar{0}}^{2}-\|e_{w,h}^{n}\|_{\bar{0}}^{2}\right)-\frac{3}{2}\left[-\left(\|e_{w,h}^{n-1}-e_{w,h}^{n+1}\|_{\bar{0}}^{2}+
       \|e_{w,h}^{n-1}\|_{\bar{0}}^{2}-\|e_{w,h}^{n+1}\|_{\bar{0}}^{2}\right)+\|e_{w,h}^{n-1}-e_{w,h}^{n}\|_{\bar{0}}^{2}+\|e_{w,h}^{n-1}\|_{\bar{0}}^{2}-
       \|e_{w,h}^{n}\|_{\bar{0}}^{2}\right]
       \end{equation*}
       \begin{equation*}
        -\frac{1}{2}\left[-\left(\|e_{w,h}^{n+1}-e_{w,h}^{n}\|_{\bar{0}}^{2}+\|e_{w,h}^{n+1}\|_{\bar{0}}^{2}-\|e_{w,h}^{n}\|_{\bar{0}}^{2}\right)+
        \left(\|e_{w,h}^{n+1}-e_{w,h}^{n-1}\|_{\bar{0}}^{2}+\|e_{w,h}^{n+1}\|_{\bar{0}}^{2}-\|e_{w,h}^{n-1}\|_{\bar{0}}^{2}\right)\right]=
       \end{equation*}
       \begin{equation}\label{50}
        2(\|e_{w,h}^{n+1}-e_{w,h}^{n}\|_{\bar{0}}^{2}-\|e_{w,h}^{n}-e_{w,h}^{n-1}\|_{\bar{0}}^{2})+\|e_{w,h}^{n+1}-e_{w,h}^{n-1}\|_{\bar{0}}^{2}.
       \end{equation}

        Combining equations $(\ref{49})$ and $(\ref{50})$, straightforward computations yield
        \begin{equation}\label{51}
        \left(3e_{w,h}^{n+1}-4e_{w,h}^{n}+e_{w,h}^{n-1},2e_{w,h}^{n+1}\right)_{\bar{0}}=4\|e_{w,h}^{n+1}-e_{w,h}^{n}\|_{\bar{0}}^{2}-\|e_{w,h}^{n+1}-
        e_{w,h}^{n-1}\|_{\bar{0}}^{2}+3\|e_{w,h}^{n+1}\|_{\bar{0}}^{2}-4\|e_{w,h}^{n}\|_{\bar{0}}^{2}+\|e_{w,h}^{n-1}\|_{\bar{0}}^{2}.
       \end{equation}

       But, $\|e_{w,h}^{n+1}-e_{w,h}^{n-1}\|_{\bar{0}}^{2}=\|(e_{w,h}^{n+1}-e_{w,h}^{n})+(e_{w,h}^{n}-e_{w,h}^{n-1})\|_{\bar{0}}^{2}\leq
       2(\|e_{w,h}^{n+1}-e_{w,h}^{n}\|_{\bar{0}}^{2}+\|e_{w,h}^{n}-e_{w,h}^{n-1}\|_{\bar{0}}^{2})$. Using this, it holds
       \begin{equation*}
        4\|e_{w,h}^{n+1}-e_{w,h}^{n}\|_{\bar{0}}^{2}-\|e_{w,h}^{n+1}-e_{w,h}^{n-1}\|_{\bar{0}}^{2}+3\|e_{w,h}^{n+1}\|_{\bar{0}}^{2}-4\|e_{w,h}^{n}\|_{\bar{0}}^{2}
        +\|e_{w,h}^{n-1}\|_{\bar{0}}^{2}\geq 2\|e_{w,h}^{n+1}-e_{w,h}^{n}\|_{\bar{0}}^{2}-2\|e_{w,h}^{n}-e_{w,h}^{n-1}\|_{\bar{0}}^{2}+
       \end{equation*}
       \begin{equation}\label{52}
        3\|e_{w,h}^{n+1}\|_{\bar{0}}^{2}-4\|e_{w,h}^{n}\|_{\bar{0}}^{2}+\|e_{w,h}^{n-1}\|_{\bar{0}}^{2}=2(\|e_{w,h}^{n+1}-e_{w,h}^{n}\|_{\bar{0}}^{2}-
        \|e_{w,h}^{n}-e_{w,h}^{n-1}\|_{\bar{0}}^{2})+3(\|e_{w,h}^{n+1}\|_{\bar{0}}^{2}-\|e_{w,h}^{n}\|_{\bar{0}}^{2})-(\|e_{w,h}^{n}\|_{\bar{0}}^{2}-
        \|e_{w,h}^{n-1}\|_{\bar{0}}^{2}).
       \end{equation}

       Substituting estimate $(\ref{52})$ into equation $(\ref{51})$, and utilizing equations $(\ref{45})$-$(\ref{47})$, we obtain
       \begin{equation*}
       2(\|e_{w,h}^{n+1}-e_{w,h}^{n}\|_{\bar{0}}^{2}-\|e_{w,h}^{n}-e_{w,h}^{n-1}\|_{\bar{0}}^{2})+3(\|e_{w,h}^{n+1}\|_{\bar{0}}^{2}-\|e_{w,h}^{n}\|_{\bar{0}}^{2})-
       (\|e_{w,h}^{n}\|_{\bar{0}}^{2}-\|e_{w,h}^{n-1}\|_{\bar{0}}^{2})+4\sigma B(e_{w,h}^{n+1},e_{w,h}^{n+1})\leq
       \end{equation*}
       \begin{equation}\label{53}
        \frac{4\sigma^{3}}{3}\left(w_{3t}(\epsilon(t_{n+1})),e_{w,h}^{n+1}\right)_{\bar{0}}.
       \end{equation}

       It follows from the second estimate in Lemma $\ref{l1}$ that
       \begin{equation}\label{54}
        B(e_{w,h}^{n+1},e_{w,h}^{n+1})=\|e_{w,h}^{n+1}\|_{B}^{2}\geq \frac{k_{\min}}{k_{\max}}\|e_{w,h}^{n+1}\|_{\nabla}^{2}.
       \end{equation}

       Additionally, using the Poincar\'{e}-Friedrich inequality $(\ref{38})$ and equation $(\ref{13})$, simple calculations give: for $z=(u,\psi)\in V$
       \begin{equation*}
       \|z\|_{\bar{0}}^{2}=\eta\|u\|_{0}^{2}+\rho gS_{0}\|\psi\|_{0}^{2}\leq \widetilde{C}(\eta\|u\|_{f}^{2}+\rho gS_{0}\|\psi\|_{p}^{2})=
       \widetilde{C}(\frac{1}{\nu}(\nu\eta)\|u\|_{f}^{2}+\rho gk_{\max}\frac{S_{0}}{k_{\max}}\|\psi\|_{p}^{2})\leq
       \end{equation*}
       \begin{equation}\label{55}
         \widetilde{C}\max\{\nu^{-1},S_{0}k_{\max}^{-1}\}(\nu\eta\|u\|_{f}^{2}+\rho gk_{\max}\|\psi\|_{p}^{2})=\widetilde{C}_{1}\|z\|_{\nabla}^{2},
       \end{equation}
       where $\widetilde{C}_{1}=\widetilde{C}\max\{\nu^{-1},S_{0}k_{\max}^{-1}\}$. Furthermore, utilizing the Cauchy-Schwarz inequality along with estimates $(\ref{54})$ and $(\ref{55})$, direct calculations provide
       \begin{equation*}
       \frac{4\sigma^{3}}{3}\left(w_{3t}(\epsilon(t_{n+1})),e_{w,h}^{n+1}\right)_{\bar{0}}\leq \frac{4\sigma^{3}}{3}\|e_{w,h}^{n+1}\|_{\bar{0}
       }\|w_{3t}(\epsilon(t_{n+1}))\|_{\bar{0}}\leq \frac{4\sigma^{3}}{3}\sqrt{\widetilde{C}_{1}}\|e_{w,h}^{n+1}\|_{\nabla}\|w_{3t}(\epsilon(t_{n+1}))\|_{\bar{0}}=
       \end{equation*}
       \begin{equation*}
       2\left(\sqrt{3\sigma \frac{k_{\min}}{k_{\max}}}\|e_{w,h}^{n+1}\|_{\nabla}\right)\left(2\sqrt{\frac{\widetilde{C}_{1}k_{\max}}{27k_{\min}}\sigma^{5}}
       \|w_{3t}(\epsilon(t_{n+1}))\|_{\bar{0}}\right)\leq \frac{3\sigma k_{\min}}{k_{\max}}\|e_{w,h}^{n+1}\|_{\nabla}^{2}+
       \end{equation*}
       \begin{equation}\label{56}
       \frac{4\widetilde{C}_{1}k_{\max}}{27k_{\min}}\sigma^{5}\|w_{3t}(\epsilon(t_{n+1}))\|_{\bar{0}}^{2}\leq 3\sigma \|e_{w,h}^{n+1}\|_{B}^{2}+\frac{4\widetilde{C}_{1}k_{\max}}{27k_{\min}}\sigma^{5}\|w_{3t}(\epsilon(t_{n+1}))\|_{\bar{0}}^{2}.
       \end{equation}

       Substituting estimate $(\ref{56})$ into inequality $(\ref{53})$ and after simplification to get
       \begin{equation*}
       2(\|e_{w,h}^{n+1}-e_{w,h}^{n}\|_{\bar{0}}^{2}-\|e_{w,h}^{n}-e_{w,h}^{n-1}\|_{\bar{0}}^{2})+3(\|e_{w,h}^{n+1}\|_{\bar{0}}^{2}-\|e_{w,h}^{n}\|_{\bar{0}}^{2})-
       (\|e_{w,h}^{n}\|_{\bar{0}}^{2}-\|e_{w,h}^{n-1}\|_{\bar{0}}^{2})+\sigma\|e_{w,h}^{n+1}\|_{B}^{2}\leq
       \end{equation*}
       \begin{equation}\label{57}
        \frac{4\widetilde{C}_{1}k_{\max}}{27k_{\min}}\sigma^{5}\|w_{3t}(\epsilon(t_{n+1}))\|_{\bar{0}}^{2}.
       \end{equation}

       But, inequality $(\ref{57})$ holds for any integer $l$ satisfying, $1\leq l\leq n$. Summing up estimate $(\ref{57})$, for $l=1,2,...,n-1$, and rearranging terms, this results in
        \begin{equation*}
       2\|e_{w,h}^{n+1}-e_{w,h}^{n}\|_{\bar{0}}^{2}+3\|e_{w,h}^{n+1}\|_{\bar{0}}^{2}-\|e_{w,h}^{n}\|_{\bar{0}}^{2}+\sigma\underset{l=1}{\overset{n}\sum} \|e_{w,h}^{l+1}\|_{B}^{2} \leq 2\|e_{w,h}^{1}-e_{w,h}^{0}\|_{\bar{0}}^{2}+ 3\|e_{w,h}^{1}\|_{\bar{0}}^{2}-\|e_{w,h}^{0}\|_{\bar{0}}^{2}+
       \end{equation*}
       \begin{equation}\label{58}
        \frac{4\widetilde{C}_{1}k_{\max}}{27k_{\min}}\sigma^{5}\underset{l=1}{\overset{n}\sum}\|w_{3t}(\epsilon(t_{n+1}))\|_{\bar{0}}^{2}
       \leq  2\|e_{w,h}^{1}-e_{w,h}^{0}\|_{\bar{0}}^{2}+
       3\|e_{w,h}^{1}\|_{\bar{0}}^{2}-\|e_{w,h}^{0}\|_{\bar{0}}^{2}+\frac{4\widetilde{C}_{1}k_{\max}}{27k_{\min}}n\sigma^{5}\||w_{3t}|\|_{\bar{0},\infty}^{2}.
       \end{equation}

       But $\|e_{w,h}^{n}\|_{\bar{0}}^{2}=\|e_{w,h}^{n}-e_{w,h}^{n+1}+e_{w,h}^{n+1}\|_{\bar{0}}^{2}\leq 2(\|e_{w,h}^{n}-e_{w,h}^{n+1}\|_{\bar{0}}^{2}+\|e_{w,h}^{n+1}\|_{\bar{0}}^{2})$. Thus,
       \begin{equation}\label{59}
        2\|e_{w,h}^{n+1}-e_{w,h}^{n}\|_{\bar{0}}^{2}+3\|e_{w,h}^{n+1}\|_{\bar{0}}^{2}-\|e_{w,h}^{n}\|_{\bar{0}}^{2}=\underset{\geq0}{\underbrace{2(\|e_{w,h}^{n}-e_{w,h}^{n+1}\|_{\bar{0}}^{2}+
       \|e_{w,h}^{n+1}\|_{\bar{0}}^{2})-\|e_{w,h}^{n}\|_{\bar{0}}^{2}}}+\|e_{w,h}^{n+1}\|_{\bar{0}}^{2} \geq \|e_{w,h}^{n+1}\|_{\bar{0}}^{2}.
       \end{equation}

       Since $\sigma=\frac{T}{N}$, so $n\sigma\leq T$. A combination of estimates $(\ref{58})$ and $(\ref{59})$ yields
       \begin{equation*}
       \|e_{w,h}^{n+1}\|_{\bar{0}}^{2}+\sigma\underset{l=1}{\overset{n}\sum} \|e_{w,h}^{l+1}\|_{B}^{2} \leq  2\|e_{w,h}^{1}-e_{w,h}^{0}\|_{\bar{0}}^{2}+
       3\|e_{w,h}^{1}\|_{\bar{0}}^{2}-\|e_{w,h}^{0}\|_{\bar{0}}^{2}+\frac{T\widetilde{C}_{1}k_{\max}}{9k_{\min}}\sigma^{4}\||w_{3t}|\|_{\bar{0},\infty}^{2}\leq
       \end{equation*}
       \begin{equation}\label{60}
        5\|e_{w,h}^{1}\|_{\bar{0}}^{2}+3\|e_{w,h}^{0}\|_{\bar{0}}^{2}+ \frac{4T\widetilde{C}_{1}k_{\max}}{24k_{\min}}\sigma^{4}\||w_{3t}|\|_{\bar{0},\infty}^{2}.
       \end{equation}

       It follows from the initial conditions $(\ref{5})$, $(\ref{32})$ and $(\ref{34})$ together with estimates $(\ref{36})$-$(\ref{37})$ that
       \begin{equation*}
       \|e_{w,h}^{1}\|_{\bar{0}}^{2}=\|w_{h}^{1}-w^{1}\|_{\bar{0}}^{2}=\|\mathbb{P}_{h}\overline{w}^{1}-w^{1}\|_{\bar{0}}^{2}=\|\mathbb{P}_{h}\overline{w}^{1}-
       \mathbb{P}_{h}w^{1}+\mathbb{P}_{h}w^{1}-w^{1}\|_{\bar{0}}^{2}\leq 2\|\mathbb{P}_{h}\overline{w}^{1}-\mathbb{P}_{h}w^{1}\|_{\bar{0}}^{2}+
       \end{equation*}
       \begin{equation}\label{61}
        2\|\mathbb{P}_{h}w^{1}-w^{1}\|_{\bar{0}}^{2}=2(\|\overline{w}^{1}-w^{1}\|_{\bar{0}}^{2}+\|\mathbb{P}_{h}w^{1}-w^{1}\|_{\bar{0}}^{2})\leq \frac{\sigma^{4}}{2}\||w_{2t}|\|_{\bar{0},\infty}^{2}+2C_{1}^{2}h^{2d+2}\|w_{0}\|_{V^{d+1}}^{2},
       \end{equation}
       \begin{equation}\label{62}
       \|e_{w,h}^{0}\|_{\bar{0}}^{2}=\|w_{h}^{0}-w^{0}\|_{\bar{0}}^{2}=\|\mathbb{P}_{h}w^{0}-w^{0}\|_{\bar{0}}^{2}\leq C_{1}^{2}h^{2d+2}\|w\|_{V^{d+1}}^{2}.
       \end{equation}

       Substituting inequalities $(\ref{61})$ and $(\ref{62})$ into estimate $(\ref{60})$ and rearranging terms, this gives
       \begin{equation*}
       \|e_{w,h}^{n+1}\|_{\bar{0}}^{2}+\sigma\underset{l=1}{\overset{n}\sum} \|e_{w,h}^{l+1}\|_{B}^{2} \leq \left(\frac{5}{2}\||w_{2t}|\|_{\bar{0},\infty}^{2}+
       \frac{4T\widetilde{C}_{1}k_{\max}}{27k_{\min}}\||w_{3t}|\|_{\bar{0},\infty}^{2}\right)\sigma^{4}+C_{1}^{2}\left(3\|w_{0}\|_{V^{d+1}}^{2}
       +10\|w^{1}\|_{V^{d+1}}^{2}\right)h^{2d+2}
       \end{equation*}
       \begin{equation}\label{62a}
       \leq\max\left\{\frac{5}{2}\||w_{2t}|\|_{\bar{0},\infty}^{2}+\frac{4T\widetilde{C}_{1}k_{\max}}{27k_{\min}}\||w_{3t}|\|_{\bar{0},\infty}^{2}, \text{\,\,}C_{1}^{2}\left(3\|w_{0}\|_{V^{d+1}}^{2}+10\|w^{1}\|_{V^{d+1}}^{2}\right)\right\}(\sigma^{2}+h^{d+1})^{2},
       \end{equation}
       which implies
       \begin{equation*}
       \|e_{w,h}^{n+1}\|_{\bar{0}}^{2}\leq \max\left\{\frac{5}{2}\||w_{2t}|\|_{\bar{0},\infty}^{2}+\frac{4T\widetilde{C}_{1}k_{\max}}{27k_{\min}}
       \||w_{3t}|\|_{\bar{0},\infty}^{2}, \text{\,\,}C_{1}^{2}\left(3\|w_{0}\|_{V^{d+1}}^{2}+10\|w^{1}\|_{V^{d+1}}^{2}\right)\right\}(\sigma^{2}+h^{d+1})^{2}-
       \end{equation*}
       \begin{equation*}
       \sigma\underset{l=1}{\overset{n}\sum}\|e_{w,h}^{l+1}\|_{B}^{2}
       \end{equation*}

       The square root in both sides of this estimate results in
       \begin{equation*}
       \|e_{w,h}^{n+1}\|_{\bar{0}} \leq \left(\max\left\{\frac{5}{2}\||w_{2t}|\|_{\bar{0},\infty}^{2}+\frac{4T\widetilde{C}_{1}k_{\max}}{27k_{\min}}
       \||w_{3t}|\|_{\bar{0},\infty}^{2}, \text{\,\,}C_{1}^{2}\left(3\|w_{0}\|_{V^{d+1}}^{2}+10\|w^{1}\|_{V^{d+1}}^{2}\right)\right\}(\sigma^{2}+h^{d+1})^{2}-\right.
       \end{equation*}
       \begin{equation}\label{63}
       \left.\sigma\underset{l=1}{\overset{n}\sum}\|e_{w,h}^{l+1}\|_{B}^{2}\right)^{\frac{1}{2}}.
       \end{equation}

       But $\|e_{w,h}^{n+1}\|_{\bar{0}}=\|w_{h}^{n+1}-w^{n+1}\|_{\bar{0}}\geq \|w_{h}^{n+1}\|_{\bar{0}}-\|w^{n+1}\|_{\bar{0}}$. This fact combined with estimate $(\ref{63})$ imply
       \begin{equation*}
       \|w_{h}^{n+1}\|_{\bar{0}}\leq \|w^{n+1}\|_{\bar{0}}+\left(\max\left\{\frac{5}{2}\||w_{2t}|\|_{\bar{0},\infty}^{2}+\frac{4T\widetilde{C}_{1}k_{\max}}{27k_{\min}}
       \||w_{3t}|\|_{\bar{0},\infty}^{2}, \text{\,\,}C_{1}^{2}\left(3\|w_{0}\|_{V^{d+1}}^{2}+10\|w^{1}\|_{V^{d+1}}^{2}\right)\right\}\times\right.
       \end{equation*}
       \begin{equation*}
       \left.(\sigma^{2}+h^{d+1})^{2}-\sigma\underset{l=1}{\overset{n}\sum}\|e_{w,h}^{l+1}\|_{B}^{2}\right)^{\frac{1}{2}}.
       \end{equation*}

       Squared both sides of this estimate, observing that $(a+b)^{2}\leq 2(a^{2}+b^{2})$, for two real numbers $a$ and $b$, and rearranging terms provide
       \begin{equation*}
       \|w_{h}^{n+1}\|_{\bar{0}}^{2}+2\sigma\underset{l=1}{\overset{n}\sum}\|e_{w,h}^{l+1}\|_{B}^{2}\leq 2\|w^{n+1}\|_{\bar{0}}^{2}+2\max\left\{\frac{5}{2}
       \||w_{2t}|\|_{\bar{0},\infty}^{2}+\frac{4T\widetilde{C}_{1}k_{\max}}{27k_{\min}}\||w_{3t}|\|_{\bar{0},\infty}^{2},\right.
       \end{equation*}
       \begin{equation}\label{64}
       \left.C_{1}^{2}\left(3\|w_{0}\|_{V^{d+1}}^{2}+10\|w^{1}\|_{V^{d+1}}^{2}\right)\right\}(\sigma^{2}+h^{d+1})^{2}.
       \end{equation}

       Taking the maximum over $t\in[0,\text{\,}T]$ in both inequalities $(\ref{62a})$ and $(\ref{64})$, the proof of Theorem $\ref{t1}$ is completed thanks to estimates $(\ref{41a})$, $(\ref{42})$ and $(\ref{55})$.
        \end{proof}

        \begin{remark}
        It's worth mentioning that the proof of unconditional stability and error estimates of the proposed approach $(\ref{30})$-$(\ref{32})$ and $(\ref{34})$ in a computed solution of the initial-boundary value problem $(\ref{1})$-$(\ref{6})$ has been established with neither the use of Gronwall inequality nor the addition of stabilization terms as widely discussed in the literature \cite{5en1,jkmt}.
        \end{remark}

         \section{Numerical experiments}\label{sec4}

       \text{\,\,\,\,\,\,\,\,\,\,} In this section, we perform some numerical examples to verify the unconditional stability, second-order accurate in time and spatial convergence of order $O(h^{d+1})$, where $d=3$, of the developed computational technique $(\ref{30})$-$(\ref{32})$ and $(\ref{34})$ for solving the nonstationary coupled groundwater-surface water problem $(\ref{1})$-$(\ref{6})$. Three tests dealing with some practical applications and corresponding to different values of $S_{0}$ and $\eta$ are discussed using the $L^{\infty}$-norms ($\||\cdot|\|_{\bar{0},\infty}$ and $\||\cdot|\|_{0,\infty}$) defined in equations $(\ref{11})$ and $(\ref{13})$: (i) $S_{0}=10^{-3}$ and $\eta=10^{-2}$, (ii) $S_{0}=10^{-7}$ and $\eta=10^{-2}$ and (iii) $S_{0}=10^{-10}$ and $\eta=10^{-1}$. Additionally, we analyze the stability and error estimates of the proposed computational technique using various time steps $\sigma=2^{-l}$, for $l\in\{4,5,6,7,8\}$, and grid spaces $h\in\{\frac{1}{2^{l}}, {\,\,for\,\,}l=2,3,4,5,6\}$. The finite element spaces are constructed by the quadrilateral Taylor-Hood element $Q_{2}/(Q_{1}+Q_{0})$ meshes in which $Q_{2}$ is used to approximate the velocity $w$ while we use $Q_{1}+Q_{0}$ to approximate the Stokes-pressure $p$. We compute the norm of the analytical solution ($w$), the approximate one ($w_{h}$) and the error terms: $e_{w,h}=w-w_{h}$ and $e_{p,h}=p-p_{h}$, using the norm, $\||\cdot|\|_{\theta,\infty}$, defined as

       \begin{equation*}
       \||z|\|_{\theta,\infty}=\underset{1\leq n\leq N}{\max}\|z^{n}\|_{\theta},
         \end{equation*}
       where $\|\cdot\|_{\theta}$, for $\theta\in\{\bar{0},\text{\,}0\}$, are the norms defined in equations $(\ref{11})$ and $(\ref{13})$, and $z\in V$ or $L_{0}^{2}(\Omega)$. Finally, the convergence order in space, $CO(h)$, of the new algorithm is estimated utilizing the formula
         \begin{equation*}
          CO(h)=\frac{\log\left(\frac{\||z_{2h}-z|\|_{\theta,\infty}}{\||z_{h}-z|\|_{\theta,\infty}}\right)}{\log(2)}.
         \end{equation*}

         Here, $z_{h}$ and $z_{2h}$ denote the spatial errors associated with the grid sizes $h$ and $2h$, respectively, while the temporal convergence order, $CO(\sigma)$, is computed using the formula
         \begin{equation*}
          CO(\sigma)=\frac{\log\left(\frac{\||z_{2\sigma}-z|\|_{\theta,\infty}}{\||z_{\sigma}-z|\|_{\theta,\infty}}\right)}{\log(2)},
         \end{equation*}
         where $z_{2\sigma}$ and $z_{\sigma}$ represent the temporal errors corresponding to time steps $2\sigma$ and $\sigma$, respectively. The numerical computations are carried out by the use of MatLab $R2007b$.\\

         Let $\Omega_{f}=(0,\text{\,}1)\times(1,\text{\,}2)$ be the fluid flow region, $\Omega_{p}=(0,\text{\,}1)^{2}$ be the porous media flow and $T=1$ be the final time. We assume that the external force $f=0$ and source term $g_{0}=0$, the initial and boundary conditions are chosen so that the analytical solution $((v_{1},v_{2}),p,\phi)^{T}$ is given in \cite{5en1} by
        \begin{eqnarray*}
         v_{1}(x,y,t) &=& \left(x^{2}(y-1)^{2}+y\right)\cos(t),\\
         v_{2}(x,y,t) &=& \left(\frac{2}{3}x(1-y)^{3}+2-\pi\sin(\pi x)\right)\cos(t),\\
         p(x,y,t) &=& \left(2-\pi\sin(\pi x)\right)\sin(0.5\pi y)\cos(t),\\
         \phi(x,y,t) &=& \left(2-\pi\sin(\pi x)\right)\left(1-y-\cos(\pi y)\right)\cos(t).
        \end{eqnarray*}

         $\bullet$ \textbf{Test $1$}. We take $S_{0}=10^{-3}$, $\eta=10^{-2}$, $k_{min}=10^{-2}$, $k_{\max}=1$, $\nu=10^{-1}$, $g=10$, $\rho=10^{3}$ and $T=1$.\\

         \textbf{Table 1.} Exact solution $w$, approximate one $w_{h}$ and convergence order $CO(\theta)$, where $\theta\in\{h,\sigma\}$, of the proposed hybrid higher-order computational technique with $S_{0}=10^{-3}$, $\eta=10^{-2}$, varying grid space $h$ and time step $\sigma$.
           \begin{equation*}
          \begin{array}{c c}
          \text{\,new algorithm,\,\,where\,\,}\sigma=2^{-6}& \\
           \begin{tabular}{cccccccc}
            \hline
            $h$ & $\||w|\|_{\bar{0},\infty}$ & $\||w_{h}|\|_{\bar{0},\infty}$ & $\||w_{h}-w|\|_{\bar{0},\infty}$ & $CO(h)$ & $\||p_{h}-p|\|_{0,\infty}$ & $CO(h)$ & CPU(s)\\
             \hline
            $2^{-2}$ & $3.0695$ & $3.0466$ & $2.0135\times10^{-2}$ & ....   &  $9.1103\times10^{-2}$ & .... & $2.6719$\\

            $2^{-3}$ & $2.3560$ & $2.3384$ & $1.2685\times10^{-3}$ & $3.9885$ & $2.2774\times10^{-2}$ & $2.0001$ & $7.8750$\\

            $2^{-4}$ & $2.0267$ & $2.0116$ & $8.0729\times10^{-5}$ & $3.9739$ & $5.5555\times10^{-3}$ & $2.0354$ & $26.3376$\\

            $2^{-5}$ & $1.8705$ & $1.8565$ & $5.0445\times10^{-6}$ & $4.0003$ & $1.3915\times10^{-3}$ & $1.9973$ & $51.5452$\\

            $2^{-6}$ & $1.7948$ & $1.7814$ & $2.9364\times10^{-7}$ & $4.1026$ & $3.2265\times10^{-4}$ & $2.1086$ & $126.7267$\\
            \hline
          \end{tabular} &
          \end{array}
          \end{equation*}

            \begin{equation*}
          \begin{array}{c c}
          \text{\,method discussed in \cite{11zd},\,\,where\,\,}\sigma=0.01& \\
           \begin{tabular}{ccccccc}
            \hline
            $h$ & $\|v_{h}^{n}-v^{n}\|_{0}$ & $\|v_{h}^{n}-v^{n}\|_{f}$ & $\|p_{h}^{n}-p^{n}\|_{0}$ & $\|\phi_{h}-\phi\|_{0}$ & $\|\phi_{h}-\phi\|_{p}$ & CPU(s)\\
             \hline
            $2^{-1}$ & $0.260588$ & $1.50020$ & $0.84999$ & $0.154615$ & $1.37567$ & $1.045$\\

            $2^{-2}$ & $0.070374$ & $0.80750$ & $0.47252$ & $0.047351$ & $0.79330$ & $3.058$\\

            $2^{-3}$ & $0.018036$ & $0.41540$ & $0.22371$ & $0.012932$ & $0.40959$ & $10.249$\\

            $2^{-4}$ & $0.004316$ & $0.18940$ & $0.07481$ & $0.003004$ & $0.19542$ & $38.454$\\

            $2^{-5}$ & $0.001095$ & $0.09588$ & $0.03597$ & $0.001097$ & $0.10070$ & $155.891$\\
            \hline
          \end{tabular} &
          \end{array}
          \end{equation*}

            \begin{equation*}
          \begin{array}{c c}
          \text{\,new algorithm,\,\,where\,\,}h=2^{-5}& \\
           \begin{tabular}{cccccccc}
            \hline
            $\sigma$ & $\||w_{h}-w|\|_{\bar{0},\infty}$ & $CO(\sigma)$ & $\||p_{h}-p|\|_{0,\infty}$ & $CO(\sigma)$ & CPU(s)\\
             \hline
            $2^{-4}$ & $5.3705\times10^{-2}$ & ....   & $4.6279\times10^{-1}$ & ....  & $43.7813$\\

            $2^{-5}$ & $1.3767\times10^{-2}$ & $1.9638$ & $2.2845\times10^{-1}$ & $1.0185$ & $89.6063$\\

            $2^{-6}$ & $3.4566\times10^{-3}$ & $1.9938$ & $1.1243\times10^{-1}$ & $1.0228$ & $173.9095$\\

            $2^{-7}$ & $8.5326\times10^{-4}$ & $2.0183$ & $5.6688\times10^{-2}$ & $0.9879$ & $342.9462$\\

            $2^{-8}$ & $1.9794\times10^{-4}$ & $2.1079$ & $2.6281\times10^{-2}$ & $1.1090$ & $734.0870$\\
            \hline
          \end{tabular} &
          \end{array}
          \end{equation*}

           The Tables show that the developed algorithm $(\ref{30})$-$(\ref{32})$ and $(\ref{34})$ is temporal second-order convergent and spatial fourth-order accurate in the Stokes-Darcy velocities $w=(v,\phi)$, first-order accurate in time and spatial second-order convergence in the stokes pressure $p$. Furthermore, one should observe from these Tables that the proposed approach is faster and more efficient than the method discussed in \cite{11zd} for solving the considered mixed Stokes-Darcy model $(\ref{1})$-$(\ref{6})$.\\

          $\bullet$ \textbf{Test $2$}. Assume $S_{0}=10^{-7}$, $\eta=10^{-2}$ and the other parameters are given in \textbf{Test 1.}\\

          \textbf{Table 2.} Analytical solution $w$, computed one $w_{h}$ and convergence order $CO(\theta)$ of the new algorithm with $S_{0}=10^{-7}$ and $\eta=10^{-2}$, with varying time step $\sigma$ and mesh space $h$.
          \begin{equation*}
          \begin{array}{c c}
          \text{\,new algorithm,\,\,where\,\,}\sigma=2^{-6}& \\
           \begin{tabular}{cccccccc}
            \hline
            $h$ & $\||w|\|_{\bar{0},\infty}$ & $\||w_{h}|\|_{\bar{0},\infty}$ & $\||w_{h}-w|\|_{\bar{0},\infty}$ & $CO(h)$ & $\||p_{h}-p|\|_{0,\infty}$ & $CO(h)$ & CPU(s)\\
             \hline
            $2^{-2}$ & $ 0.1731$ & $0.1718$ & $7.2473\times10^{-3}$ & ....   & $6.5217\times10^{-2}$ & .... & $18.4879$\\

            $2^{-3}$ & $0.1354$ & $0.1344$ & $4.5992\times10^{-4}$ & $3.9780$ & $1.6374\times10^{-2}$ & $1.9938$ & $35.6037$\\

            $2^{-4}$ & $0.1183$ & $0.1174$ & $2.8717\times10^{-5}$ & $4.0014$ & $4.0661\times10^{-3}$ & $2.0097$ & $65.6973$\\

            $2^{-5}$ & $0.1102$ & $0.1094$ & $1.7972\times10^{-6}$ & $3.9981$ & $9.4229\times10^{-4}$ & $2.1094$ & $125.7839$\\

            $2^{-6}$ & $0.1063$ & $0.1055$ & $1.1095\times10^{-7}$ & $4.0178$ & $2.1887\times10^{-4}$ & $2.1061$ & $221.0271$\\
            \hline
          \end{tabular} &
          \end{array}
          \end{equation*}
            \begin{equation*}
          \begin{array}{c c}
          \text{\,new algorithm,\,\,where\,\,}h=2^{-5}& \\
           \begin{tabular}{ccccccccc}
            \hline
            $\sigma$ & $\||w_{h}-w|\|_{\bar{0},\infty}$ & $CO(\sigma)$ & $\||p_{h}-p|\|_{0,\infty}$ & $CO(\sigma)$ & CPU(s)\\
             \hline
            $2^{-4}$ & $6.0032\times10^{-3}$ & ....   & $2.7190\times10^{-1}$ & .... &     $48.9976$\\

            $2^{-5}$ & $1.5128\times10^{-3}$ & $1.9885$ & $1.3575\times10^{-1}$ & $1.0021$ & $97.2319$\\

            $2^{-6}$ & $3.5859\times10^{-4}$ & $2.0768$ & $6.8172\times10^{-2}$ & $0.9937$ & $172.2719$\\

            $2^{-7}$ & $9.0102\times10^{-5}$ & $1.9927$ & $3.4084\times10^{-2}$ & $1.0001$ & $318.7684$\\

            $2^{-8}$ & $2.2232\times10^{-5}$ & $2.0189$ & $1.6243\times10^{-2}$ & $1.0693$ & $689.1342$\\
            \hline
          \end{tabular} &
          \end{array}
          \end{equation*}

            \begin{equation*}
          \begin{array}{c c}
          \text{\,method discussed in \cite{11zd},\,\,where\,\,}h=2^{-5}& \\
           \begin{tabular}{ccccccc}
            \hline
            $\sigma$ & $\|v_{h}^{n}-v^{n}\|_{0}$ & $\|v_{h}^{n}-v^{n}\|_{f}$ & $\|p_{h}^{n}-p^{n}\|_{0}$ & $\|\phi_{h}-\phi\|_{0}$ & $\|\phi_{h}-\phi\|_{p}$ & CPU(s)\\
             \hline
            $2^{-1}$ & $0.006344$ & $0.115244$ & $0.133226$ & $0.015296$ & $0.129335$ & $13.86$\\

            $2^{-2}$ & $0.003558$ & $0.102052$ & $0.085369$ & $0.008399$ & $0.109712$ & $28.05$\\

            $2^{-3}$ & $0.001974$ & $0.097474$ & $0.054458$ & $0.004556$ & $0.103108$ & $56.09$\\

            $2^{-4}$ & $0.001281$ & $0.096224$ & $0.040079$ & $0.002566$ & $0.101279$ & $112.39$\\

            $2^{-5}$ & $0.001089$ & $0.095932$ & $0.035125$ & $0.001577$ & $0.100810$ & $223.82$\\
            \hline
          \end{tabular} &
          \end{array}
          \end{equation*}

           We observe from the Tables that the proposed computational technique $(\ref{30})$-$(\ref{32})$ and $(\ref{34})$ is second-order convergent in time and spatial fourth-order accurate in the coupling Stokes-Darcy velocities $w=(v,\phi)$ while it's temporal first-order accurate and spatial second-order convergent in the Stokes pressure $p$. In addition, the Tables suggest that the proposed hybrid higher-order interpolation/finite element scheme is faster and more accurate than the method analyzed in \cite{11zd}.\\

          $\bullet$ \textbf{Test $3$}. We take $S_{0}=10^{-10}$ and $\eta=10^{-1}$. The parameters $\nu$, $g$, $\rho$, $k_{\max}$, $k_{\min}$ and $T$ are provided in \textbf{Tests 1 $\&$ 2.}\\

          \textbf{Table 3.} Analytical solution $w$, numerical one $w_{h}$ and convergence order $CO(\theta)$, of the developed numerical approach with $S_{0}=10^{-10}$ and $\eta=10^{-1}$, with different time step $\sigma$ and mesh space $h$.
          \begin{equation*}
          \begin{array}{c c}
          \text{\,new algorithm,\,\,where\,\,}\sigma=2^{-6}& \\
           \begin{tabular}{cccccccc}
            \hline
            $h$ & $\||w|\|_{\bar{0},\infty}$ & $\||w_{h}|\|_{\bar{0},\infty}$ & $\||w_{h}-w|\|_{\bar{0},\infty}$ & $CO(h)$ & $\||p_{h}-p|\|_{0,\infty}$ & $CO(h)$ & CPU(s)\\
             \hline
            $2^{-2}$ & $0.5387$ & $0.5347$ & $8.8453\times10^{-3}$  & ....   & $2.6639\times10^{-2}$ & .... &   $17.5222$\\

            $2^{-3}$ & $0.4217$ & $0.4185$ & $5.4905\times10^{-4}$  & $4.0099$ & $6.6713\times10^{-3}$ & $1.9975$ & $32.1980$\\

            $2^{-4}$ & $0.3686$ & $0.3658$ & $3.4833\times10^{-5}$  & $3.9784$ & $1.6932\times10^{-3}$ & $1.9782$ & $55.7796$\\

            $2^{-5}$ & $0.3435$ & $0.3410$ & $2.1771\times10^{-6}$  & $4.0000$ & $4.2330\times10^{-4}$ & $2.0000$ & $100.1051$\\

            $2^{-6}$ & $0.3314$ & $0.3289$ & $1.3496\times10^{-7}$  & $4.0118$ & $9.8036\times10^{-5}$ & $2.1103$ & $193.8695$\\
            \hline
          \end{tabular} &
          \end{array}
          \end{equation*}
            \begin{equation*}
          \begin{array}{c c}
          \text{\,new algorithm,\,\,where\,\,}h=2^{-5}& \\
           \begin{tabular}{ccccccccc}
            \hline
            $\sigma$ & $\||w_{h}-w|\|_{\bar{0},\infty}$ & $CO(\sigma)$ & $\||p_{h}-p|\|_{0,\infty}$ & $CO(\sigma)$ & CPU(s)\\
             \hline
            $2^{-4}$ & $1.8394\times10^{-2}$ & ....   & $3.1142\times10^{-1}$ & .... &   $51.2117$\\

            $2^{-5}$ & $4.6343\times10^{-3}$ & $1.9888$ & $1.6778\times10^{-1}$ & $0.8923$ & $99.5038$\\

            $2^{-6}$ & $1.1605\times10^{-3}$ & $1.9976$ & $8.9123\times10^{-2}$ & $0.9127$ & $189.8990$\\

            $2^{-7}$ & $2.8736\times10^{-4}$ & $2.0138$ & $4.4881\times10^{-2}$ & $0.9897$ & $336.0233$\\

            $2^{-8}$ & $6.6317\times10^{-5}$ & $2.1154$ & $2.2493\times10^{-2}$ & $0.9966$ & $689.6456$\\
            \hline
          \end{tabular} &
          \end{array}
          \end{equation*}

          \begin{equation*}
          \begin{array}{c c}
          \text{\,method analyzed in \cite{qhpl},\,\,errors\,\,}& \\
           \begin{tabular}{cccccc}
            \hline
            $\sigma=h$ & $\||e_{v}|\|_{0,0}$ & $\||e_{v}|\|_{0,1}$ & $\||e_{\phi}|\|_{0,0}$ & $\||e_{\phi}|\|_{0,1}$ & $\||e_{p}|\|_{0,0}$ \\
             \hline
            $1/10$ & $0.01615$   & $0.506002$ & $0.0146238$  & $0.551755$ & $0.138243$ \\

            $1/16$ & $0.00652393$ & $0.311263$& $0.00599802$  & $0.359655$ & $0.0637115$ \\

            $1/22$ & $0.00351853$ & $0.22917$ & $0.00324735$  & $0.268314$ & $0.04083$ \\

            $1/28$ & $0.00218086$ & $0.176397$& $0.00202875$  & $0.211693$ & $0.0260884$ \\

            $1/34$ & $0.00149633$ & $0.148517$ & $0.0013883$  & $0.177249$ & $0.0184629$ \\
            \hline
          \end{tabular} &
          \end{array}
          \end{equation*}
           \begin{equation*}
          \begin{array}{c c}
          \text{\,method analyzed in \cite{qhpl},\,\,convergence\,\,order}& \\
           \begin{tabular}{cccccc}
            \hline
            $\sigma=h$ & $r_{v,0}$ & $r_{v,1}$ & $r_{\phi,0}$ & $r_{\phi,1}$ & $r_{p,0}$ \\
             \hline
            $1/10$ & ---  & --- & ---  & --- & --- \\

            $1/16$ & $1.92895$ & $1.03382$  & $1.8962$   & $0.910541$ & $1.64818$ \\

            $1/22$ & $1.93885$ & $0.961447$ & $1.92678$  & $0.920035$ & $1.39722$ \\

            $1/28$ & $1.98342$ & $1.08527$  & $1.95063$  & $0.982834$ & $1.85737$ \\

            $1/34$ & $1.94019$ & $0.886068$ & $1.9538$  & $0.914617$ & $1.78066$ \\
            \hline
          \end{tabular} &
          \end{array}
          \end{equation*}

          Also in this Test, the Tables indicate that the proposed numerical approach $(\ref{30})$-$(\ref{32})$ and $(\ref{34})$ is temporal second-order convergent and spatial fourth-order convergent in the coupling Stokes-Darcy velocities $w=(v,\phi)$ whereas the approach is first-order accurate in time and second-order convergent in space in the Stokes pressure $p$. In addition, the computational results show that the developed procedure is faster and more efficient than the method discussed in \cite{qhpl}.\\

           Finally, it follows from \textbf{Tables} $1$-$3$ and Figures $\ref{fig2}$-$\ref{fig4}$ that the developed hybrid higher-order interpolation/finite element technique applied to the evolutionary coupled groundwater-surface water problem $(\ref{1})$-$(\ref{6})$ is unconditionally stable, temporal second-order accurate and spatial fourth-order convergent in the Stokes-Darcy velocities $w=(v,\phi)$ while the new algorithm is temporal first-order convergent and spatial second-order accurate in the Stokes pressure $p$, faster and more efficient than a broad range of numerical methods studied in the literature \cite{11zd,5en1,en3,qhpl,11zd} for solving the initial-boundary value problem $(\ref{1})$-$(\ref{6})$.

         \section{General conclusion and future works}\label{sec5}

         \text{\,\,\,\,\,\,\,\,\,\,} In this paper, we have developed a hybrid higher-order unconditionally stable interpolation/finite element method $(\ref{30})$-$(\ref{32})$ and $(\ref{34})$ for solving the unsteady mixed Stokes-Darcy model $(\ref{1})$-$(\ref{6})$ associated with a non diagonal symmetric positive definite hydraulic conductivity tensor $\mathcal{K}$. Both stability and error estimates of the proposed numerical approach have been analyzed using the $L^{\infty}(0,T;\text{\,}L^{2})$-norm. The theory has suggested that the new algorithm is unconditionally stable, second-order accurate in time and spatial convergent of order $O(h^{d+1})$ in the Stokes-Darcy velocities $w=(v,\phi)$ for any values of the volumetric porosity $\eta$ and specific mass storativity $S_{0}$ whereas the numerical experiments have confirmed the theoretical studies and have indicated that the new computational technique is temporal first-order convergent and spatial second-order accurate in the Stokes pressure $p$. Specifically, \textbf{Tables} $1$-$3$ and \textbf{Figures} $\ref{fig2}$-$\ref{fig4}$ confirm the unconditional stability and accuracy of the developed hybrid higher-order numerical scheme $(\ref{30})$-$(\ref{32})$ and $(\ref{34})$. Additionally, both theoretical and computational results show that the new algorithm is faster and more efficient than a broad range of methods deeply analyzed in the literature \cite{11zd,5en1,en3,qhpl,11zd}, in an approximate solution of the given evolutionary coupled groundwater-surface water problem $(\ref{1})$-$(\ref{3})$ and can be considered as a robust tool for solving general systems of multi-physic model. Our future works will develop more efficient computational techniques for solving the evolutionary mixed Stokes-Darcy problem together with some nonlinear multi-physic problems arising in geological structural deformation such as the three-dimensional system of nonlinear unsteady elastodynamic sine-Gordon model.

         \subsection*{Ethical Approval}
          Not applicable.
         \subsection*{Availability of supporting data}
          Not applicable.
         \subsection*{Declaration of Interest Statement}
          The authors declare that they have no conflict of interests.
         \subsection*{Authors' contributions}
          The whole work has been carried out by the authors.
         \subsection*{Funding}
         This work was supported and funded by the Deanship of Scientific Research at Imam Mohammad Ibn Saud Islamic University (IMSIU) (grant number IMSIU-DDRSP-NS25).
          \text{\,}\\

         \begin{figure}
         \begin{center}
          Stability analysis of the developed hybrid higher-order technique for mixed Stokes-Darcy model.
          \begin{tabular}{c c}
          \psfig{file=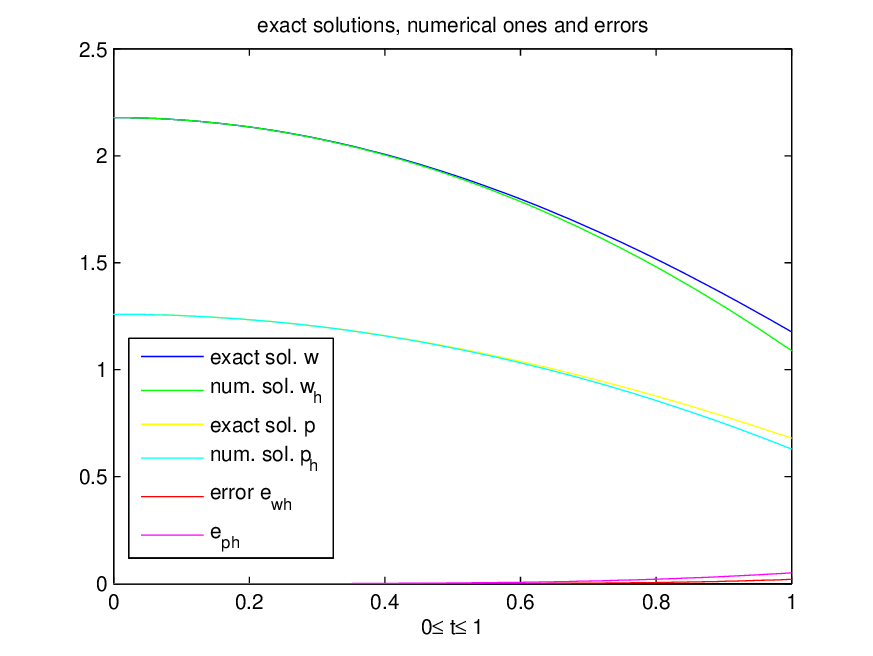,width=6cm} & \psfig{file=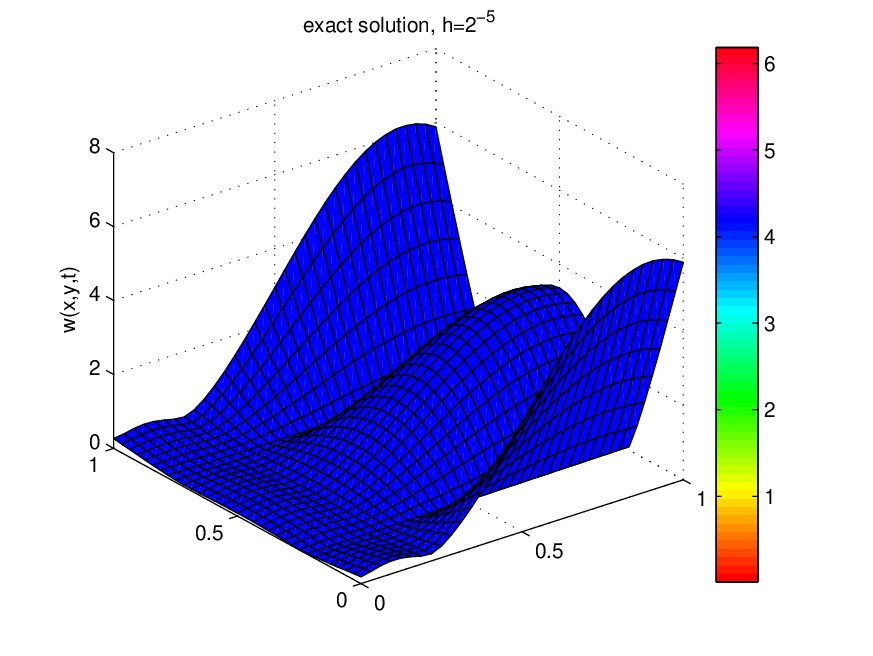,width=6cm}\\
          \psfig{file=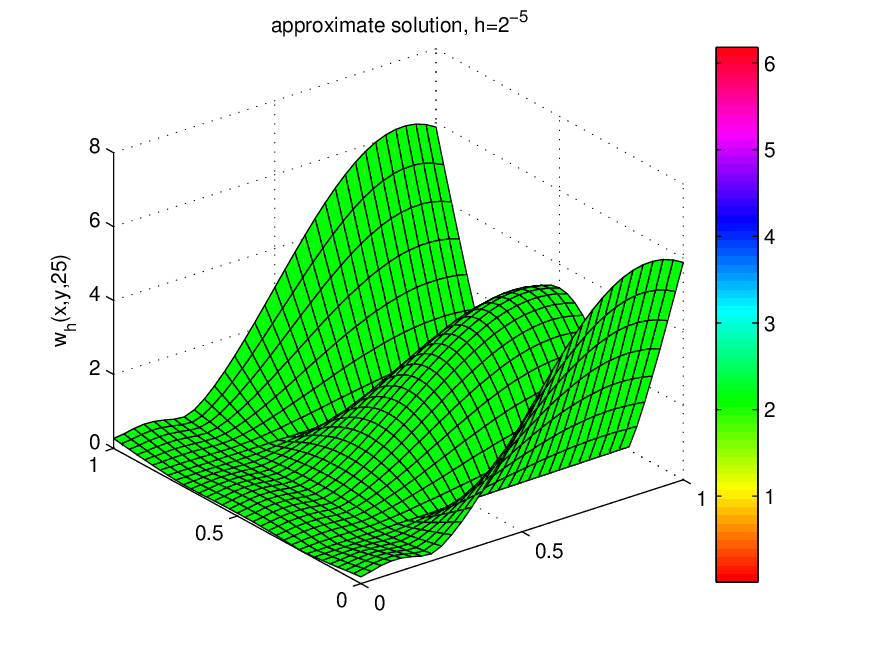,width=6cm} & \psfig{file=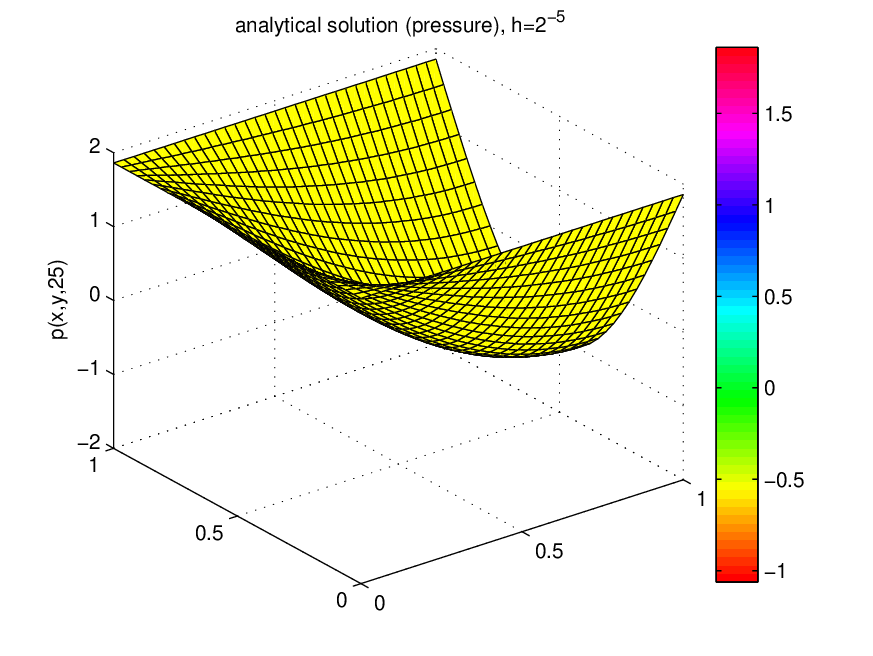,width=6cm}\\
          \psfig{file=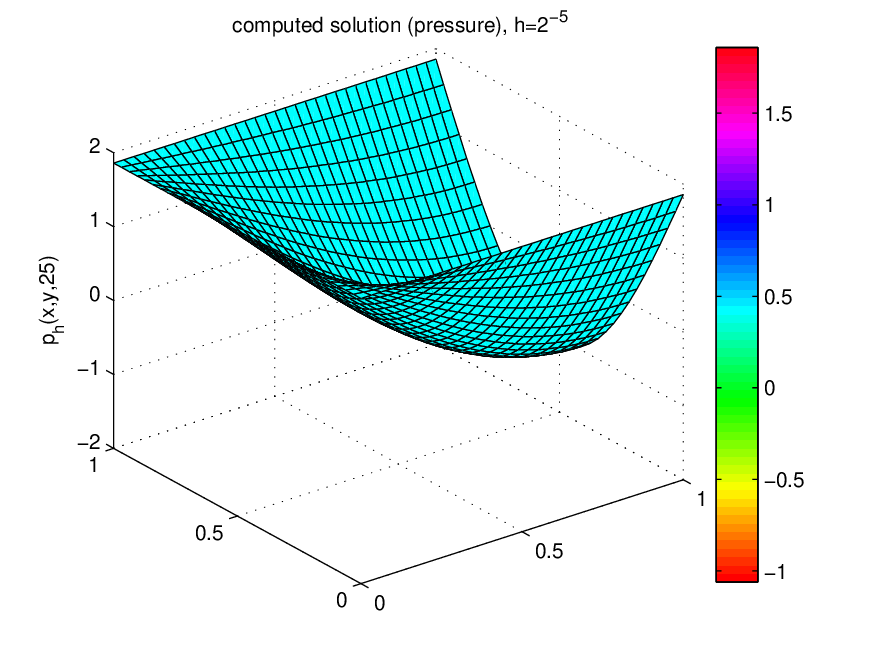,width=6cm} & \psfig{file=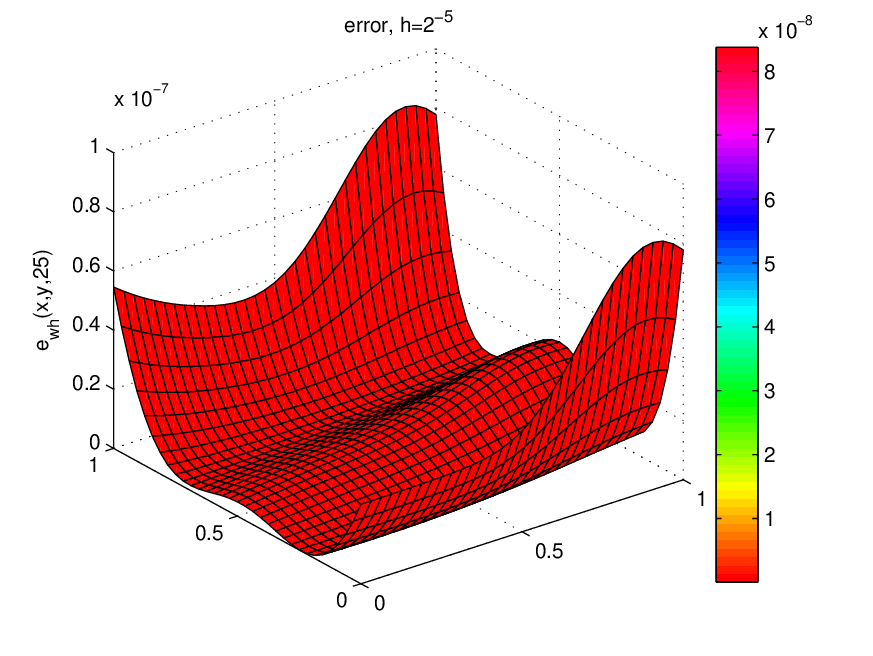,width=6cm}\\
          \psfig{file=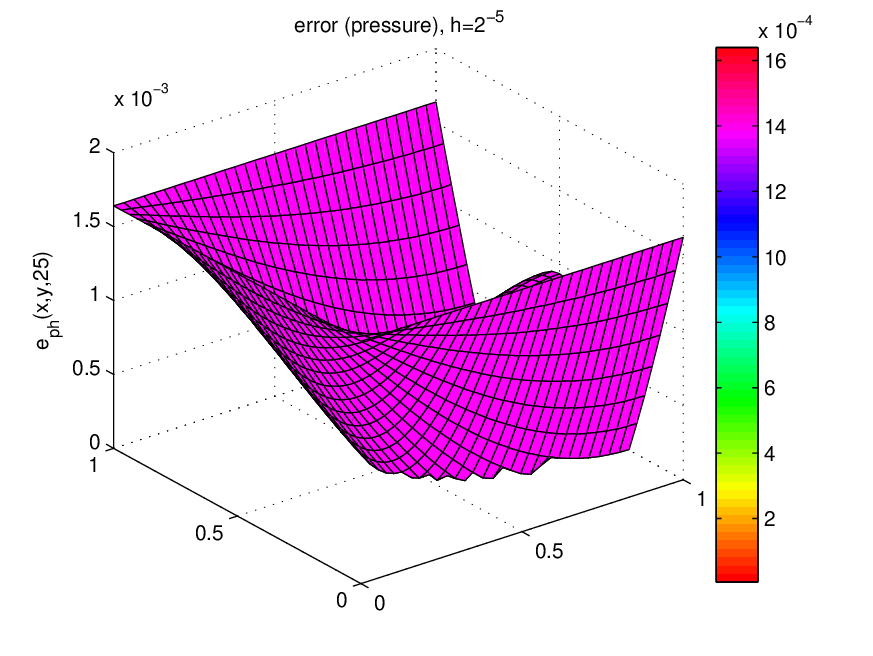,width=6cm} &                                                     \\
         \end{tabular}
        \end{center}
         \caption{corresponding to \textbf{Table 1} ($S_{0}=10^{-3}$ and $\eta=10^{-2}$)}
          \label{fig2}
          \end{figure}

          \begin{figure}
         \begin{center}
          Analysis of stability of the new algorithm applied to coupled groundwater-surface water problem.
          \begin{tabular}{c c}
          \psfig{file=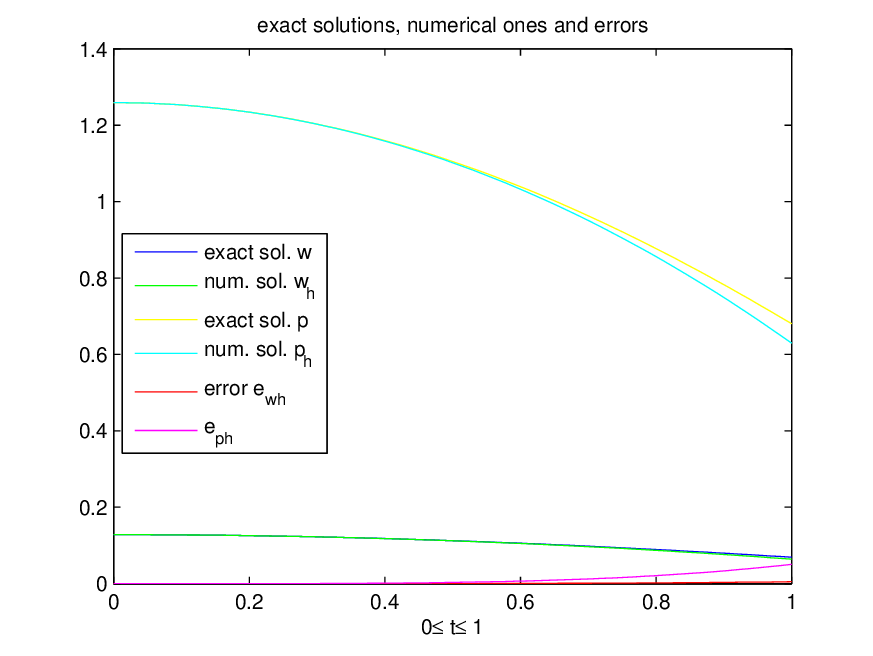,width=6cm} & \psfig{file=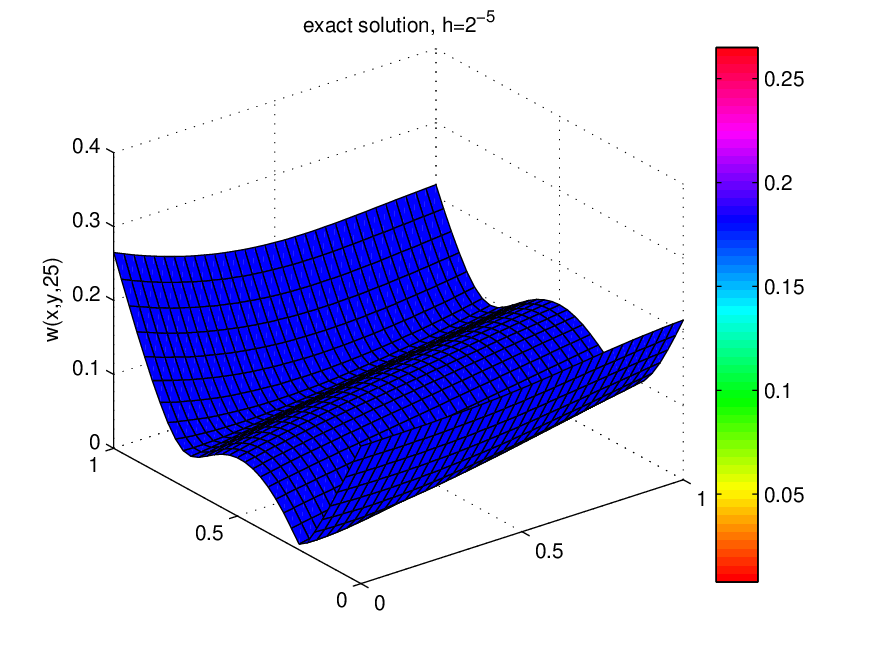,width=6cm}\\
          \psfig{file=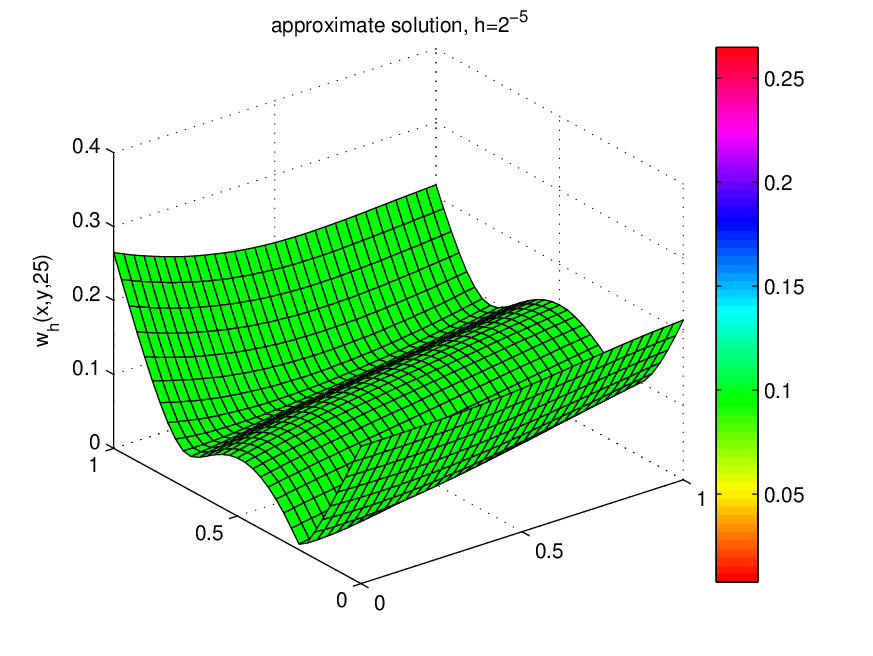,width=6cm} & \psfig{file=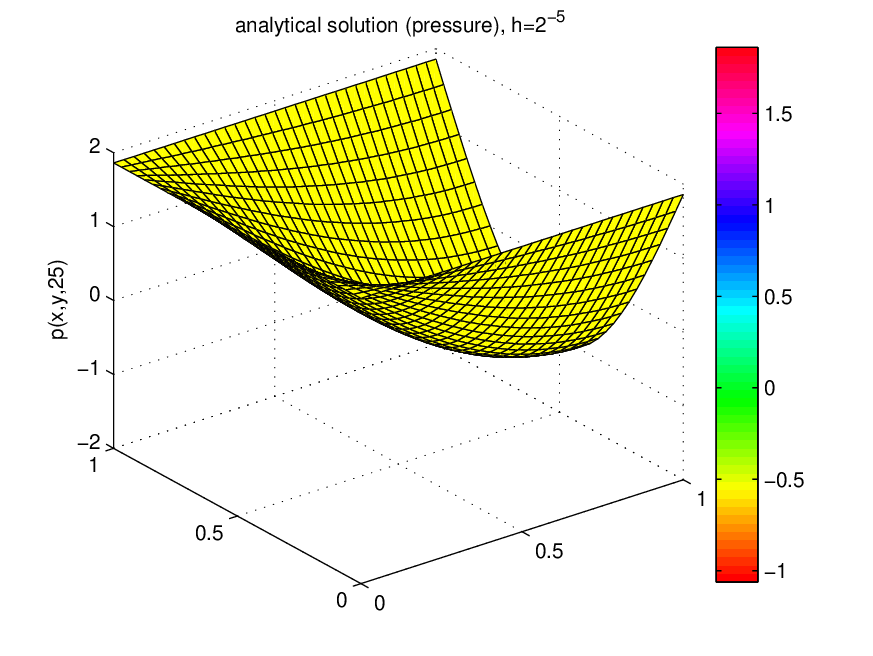,width=6cm}\\
          \psfig{file=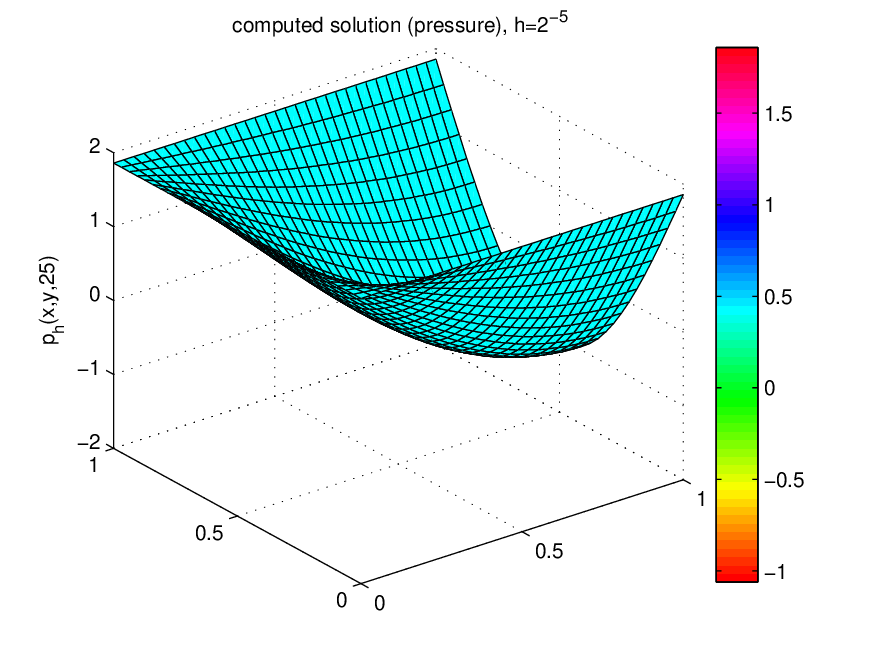,width=6cm} & \psfig{file=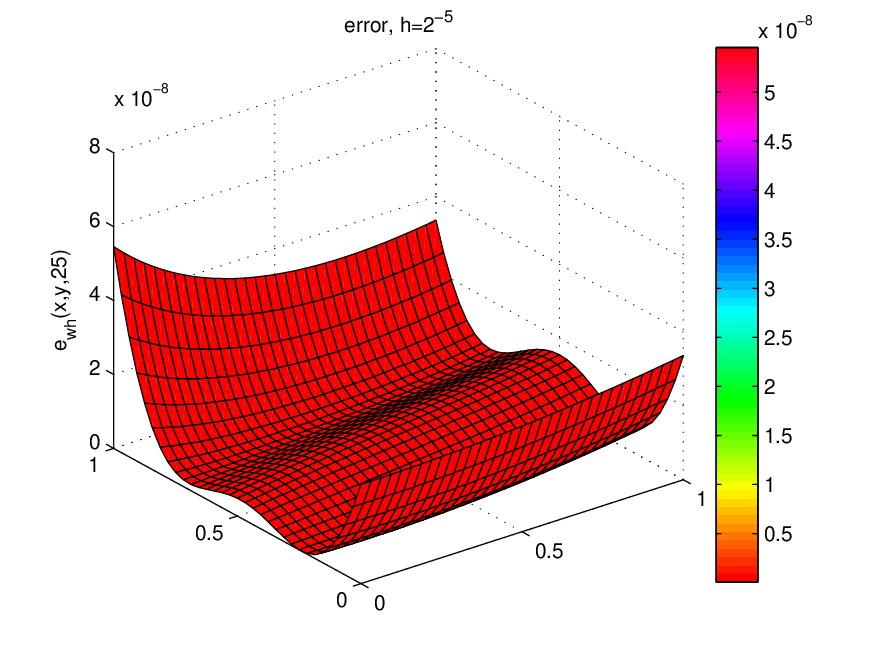,width=6cm}\\
          \psfig{file=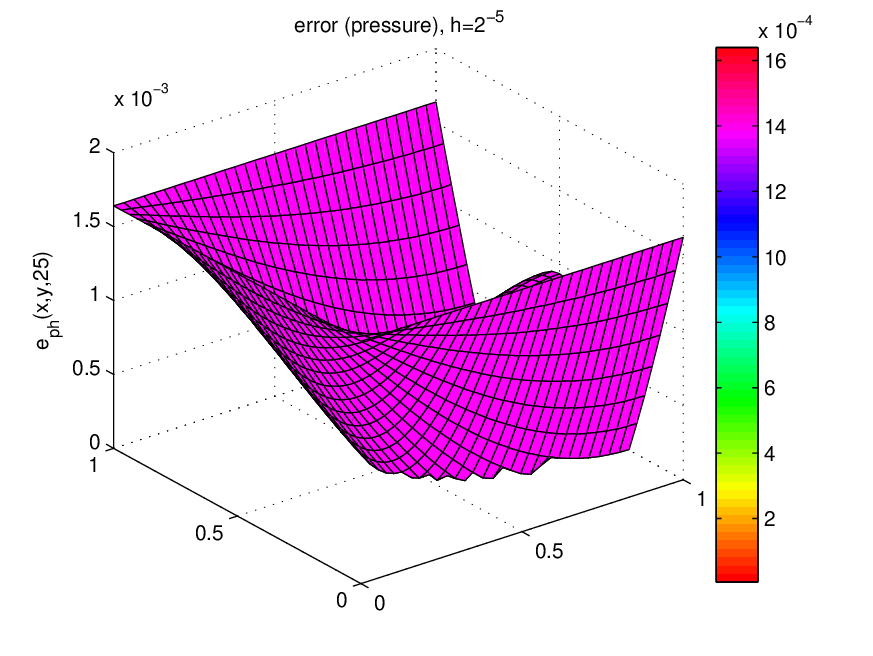,width=6cm} &  \\
         \end{tabular}
        \end{center}
        \caption{associated with \textbf{Table 2} ($S_{0}=10^{-7}$ and $\eta=10^{-2}$)}
          \label{fig3}
          \end{figure}

          \begin{figure}
         \begin{center}
          Stability analysis of the proposed hybrid higher-order interpolation/finite element approach.
          \begin{tabular}{c c}
          \psfig{file=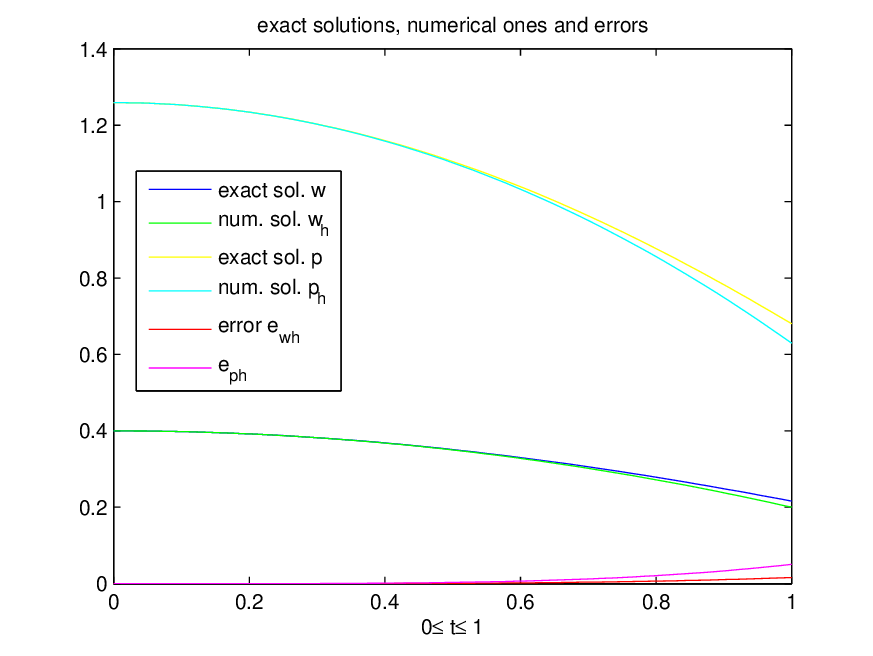,width=6cm} & \psfig{file=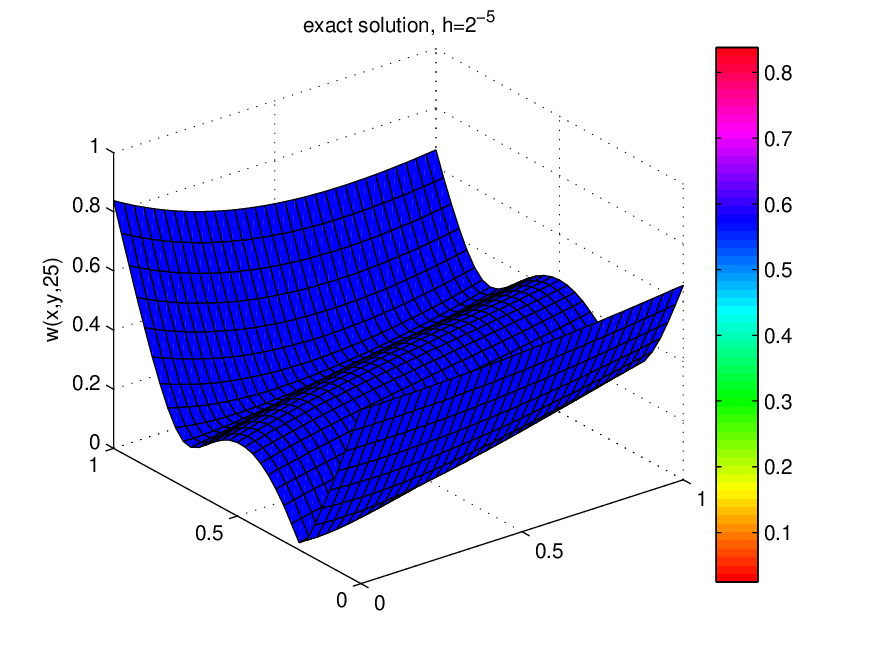,width=6cm}\\
          \psfig{file=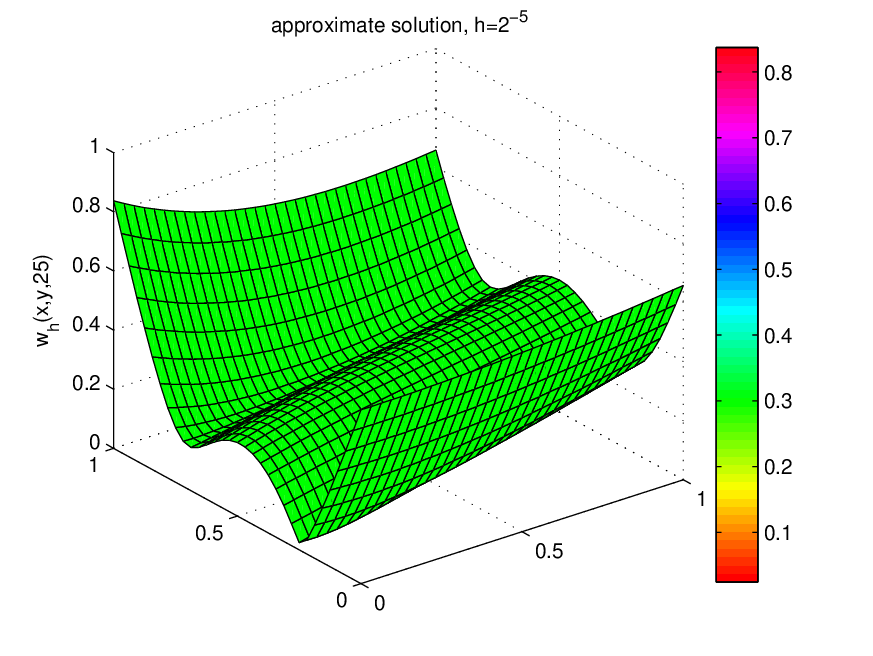,width=6cm} & \psfig{file=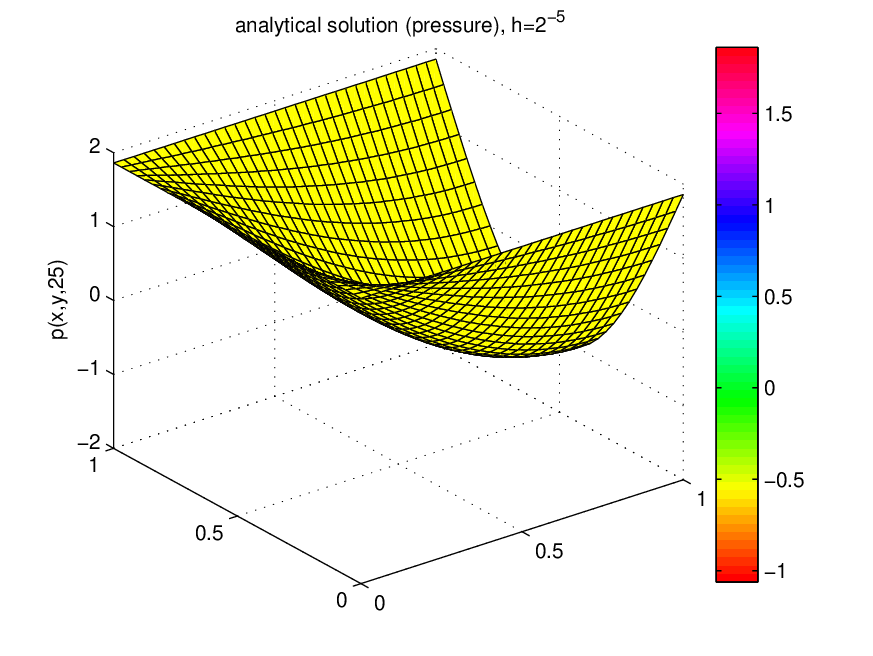,width=6cm}\\
          \psfig{file=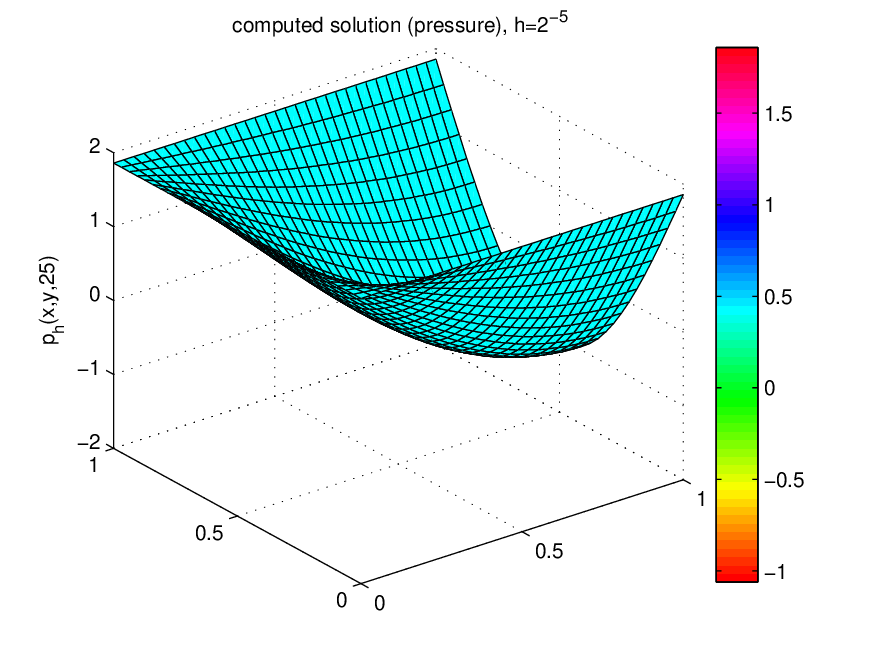,width=6cm} & \psfig{file=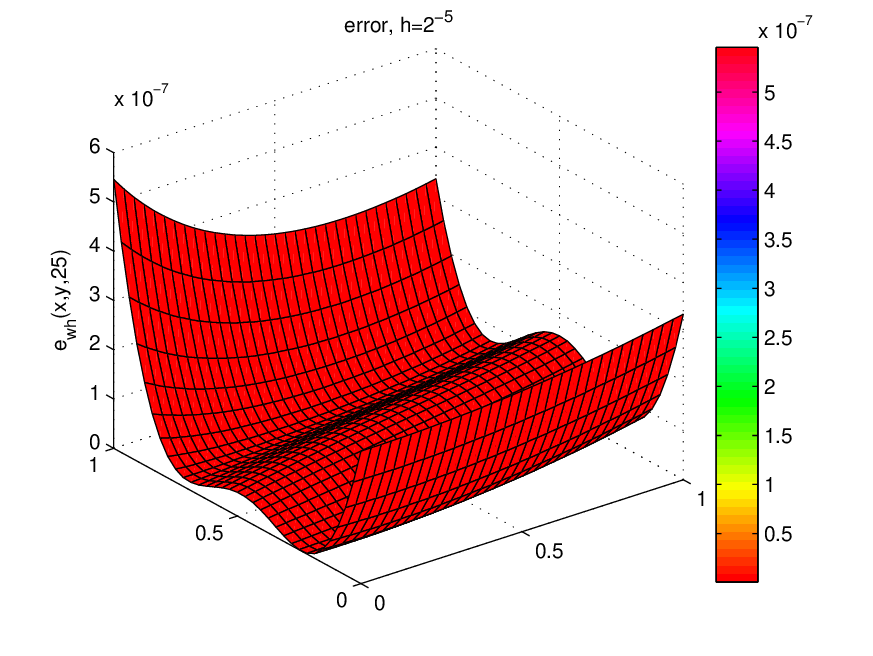,width=6cm}\\
          \psfig{file=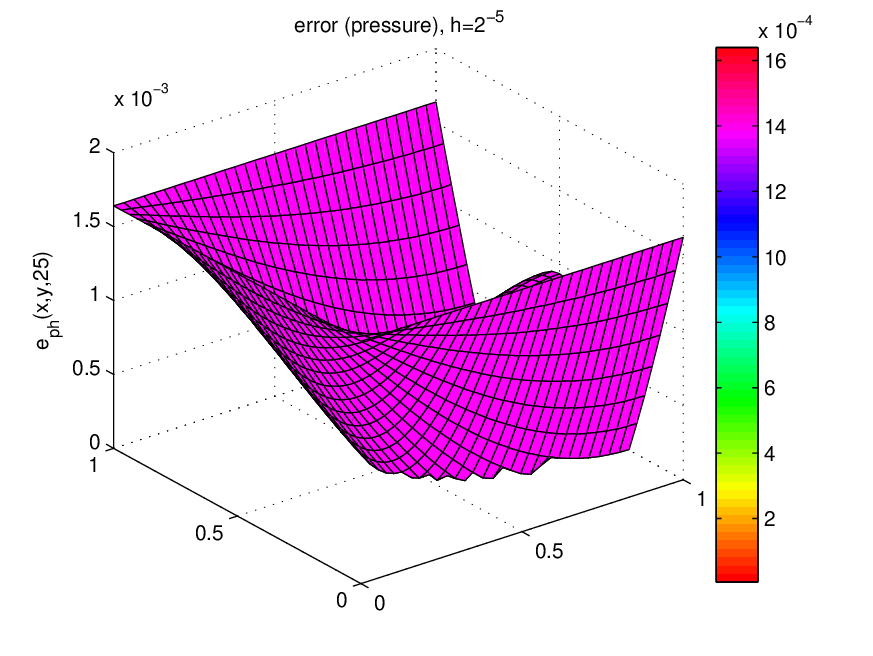,width=6cm} & \\
         \end{tabular}
        \end{center}
         \caption{corresponding to \textbf{Table 3} ($S_{0}=10^{-10}$ and $\eta=10^{-1}$)}
          \label{fig4}
          \end{figure}


\newpage
     \begin{thebibliography}{99}

     \bibitem{1en1}
     F.A. Anderson, R.H. Pletcher, J.C. Tannehill. "Computational Fluid Mechanics and Heat Transfer", second ed., Taylor and Francis, New York, $(1997)$.

     \bibitem{6en1}
     T. Arbogast, D. S. Brunson. "A computational method for approximating a Darcy-Stokes system governing a vuggy porous medium", Comput. Geosci., $11(2007)$, $207$-$218$.

    \bibitem{15lgzc}
    P. Angot, F. Boyer, F. Hubert. "Asymptotic and numerical modelling of flows in fractured porous media", ESAIM: Math. Model. Numer. Anal., $43(2)$ $(2009)$, $239$-$275$.

    \bibitem{9en1}
    J. Bear. "Hydraulics of Groundwater", McGraw-Hill, New York, $(1979)$.

    \bibitem{30jkmt}
    G. Beavers, D. Joseph. "Boundary conditions at a naturally impermeable wall", J. Fluid Mech., $30(1967)$, $197$-$207$.

    \bibitem{16lgzc}
    P. F. Antonietti, C. Facciola, A. Russo, M. Verani. "Discontinuous Galerkin approximation of flows in fractured porous media on polytopic grids", SIAM J. Sci. Comput., $41(1)$ $(2019)$, $A109$-$A138$.

    \bibitem{13en1}
    M. Cai, M. Mu, J. Xu. "Preconditioning techniques for a mixed Stokes/Darcy model in porous media applications", J. Comput. Appl. Math., $233(2009)$, $346$-$355$.

    \bibitem{26lgzc}
    Y. Cao, M. Gunzburger, X. Hu, F. Hua, X. Wang, W. Zhao. "Finite Element Approximations for Stokes-Darcy Flow with Beavers-Joseph Interface Conditions", SIAM J. Numer. Anal., $47(6)$ $(2010)$, $4239$-$4256$.

    \bibitem{14en1}
    Y. Cao, M. Gunzburger, X. He, X. Wang. "Parallel, non-iterative, multi-physics domain decomposition methods for time-dependent Stokes-Darcy systems", Math. Comp., $83(2014)$, $1617$-$1644$.

    \bibitem{2en1}
    M. Discacciati, E. Miglio, A. Quarteroni. "Mathematical and numerical models for coupling surface and groundwater flows", Appl. Numer. Math., $43(2002)$, $57$-$74$.

    \bibitem{3en1}
    M. Discacciati, A. Quarteroni. "Convergence analysis of a subdomain iterative method for the finite element approximation of the coupling of Stokes and Darcy equations", Comput. Vis. Sci., $6(2004)$, $93$-$103$.

    \bibitem{15en1}
    M. Discacciati, A. Quarteroni, A. Valli. "Robin-robin domain decomposition methods for the Stokes-Darcy coupling", SIAM J. Numer. Anal., $45(3)$ $(2007)$.

    \bibitem{7lgzc}
    A. Fumagalli, A. Scotti. "A numerical method for two-phase flow in fractured porous media with non-matching grids", Adv. water Resour., $62(2013)$, $454$-$464$.

    \bibitem{11gr}
    V. Girault, P. A. Raviart. "Finite element methods for Navier-Stokes equations", Springer-Verlag, $(1986)$.

    \bibitem{7en1}
    W. J\"{a}ger, A. Mikeli\'{c}. "On the interface boundary condition of Beavers, Joseph, and Saffman", SIAM J. Appl. Math. $60(4)$ $(2000)$, $1111$-$1127$.

    \bibitem{32jkmt}
    W. Jager, A. Mikelic. "On the boundary condition at the interface between a porous medium and a free fluid", Annali della Scuola Normale Superiore di Pisa-Classe di Scienze, Serie $4$, $23(3)$ $(1996)$, $403$-$465$.

    \bibitem{jkmt}
    N. Jiang, M. Kubacki, W. Layton, M. Moraiti, H. Tran. "A Crank-Nicolson Leapfrog stabilization: Unconditional stability and two applications", J. Comput. Appl. Math., $281(2015)$, $263$-$276$.

    \bibitem{5jkmt}
    M. Kubacki. "Uncoupling evolutionary groundwater-surface water flows using the Crank-Nicolson Leapfrog method", Numer. Methods Partial Diff. Equ., $29(2013)$, $1192$-$1216$.

    \bibitem{4en1}
    W. Layton, H. Tran, C. Trenchea. "Analysis of long time stability and errors of two partitioned methods for uncoupling evolutionary groundwater-surface water flows",
    SIAM J. Numer. Anal., $51(1)$ $(2013)$, $248$-$272$.

    \bibitem{8en1}
    E. Miglio, A. Quarteroni, F. Saleri. "Coupling of free surface and groundwater flows", Comput. Fluids, 3$2(2003)$, $73$-$83$.

    \bibitem{5en1}
    M. Mu, X. Zhu. "Decoupled schemes for a non-stationary mixed Stokes-Darcy model", Math. Comp., $79(270)$ $(2010)$, $707$-$731$.

    \bibitem{12en1}
    M. Mu, J. Xu. "A two-grid method of a mixed Stokes-Darcy model for coupling fluid flow with porous media flow", SIAM J. Numer. Anal., $45(5)$ $(2007)$, $1801$-$1813$.

    \bibitem{en1}
    E. Ngondiep. "Stability analysis of MacCormack rapid solver method for evolutionary Stokes-Darcy problem", J. Comput. Appl. Math., $345,$ $269$-$285$, $(2019)$.

    \bibitem{en2}
    E. Ngondiep. "Long time stability and convergence rate of MacCormack rapid solver method for nonstationary Stokes-Darcy problem", Comput. Math. Appl., $75,$ $3663$-$3684$, $(2018)$.

    \bibitem{en}
    E. Ngondiep. "An efficient numerical approach for solving three-dimensional Black-Scholes equation with stochastic volatility", Math. Meth. Appl. Sci., $(2024)$, $1$-$21$, Doi: $10.1002$/mma.$10576$.

    \bibitem{en3}
    E. Ngondiep. "Long time unconditional stability of a two-level hybrid method for nonstationary incompressible Navier-Stokes equations", J. Comput. Appl. Math., $345,$ $501$-$514$, $(2019)$.

    \bibitem{en5}
    E. Ngondiep. "A posteriori error estimates of MacCormack rapid solver method for nonstationary incompressible Navier-Stokes equations", J. Comput. Appl. Math., $438(2024)$, $115569$.

    \bibitem{en4}
    E. Ngondiep. "An efficient high-order weak Galerkin finite element approach for Sobolev equation with variable matrix coefficients", Comput. Math. Appl., $180$ $(2025)$ $279$-$298$.

    \bibitem{qhpl}
    Y. Qin, Y. Hou, W. Pei, J. Li. "A variable time-stepping algorithm for the unsteady Stokes/Darcy model", J. Comput. Appl. Math., $394(2021)$, $113521$.

    \bibitem{pgs}
    P. G. Saffman. "On the boundary condition at the interface of a porous medium", Studies in Applied Mathematics, Vol. $1,$ No. $2,$
    pp. $93$-$101,$  $(1971).$

    \bibitem{37lgzc}
    J. H. V. Schuppen. "Filtering, prediction and smoothing for counting process observations, a martingale approach", SIAM J. Appl. Math., $32(1977)$, $552$-$570$.

    \bibitem{26jkmt}
    L. Shan, H. Zheng, W. Layton. "A decoupling method with different subdomain time steps for the nonstationary Stokes-Darcy model", Numer. Methods Partial Diff. Equ., $29(2013)$, $549$-$583$.

    \bibitem{11zd}
    L. Shon, H. B. Zheng, W. J. Layton. "A decoupling method with different subdomain time steps for the nonstationary Stokes-Darcy model", Numer. Methods Partial Diff. Equ., Vol. $29,$ pp $549$-$583$  $(2013).$

    \bibitem{jv}
    J. Verwer. "Convergence and component splitting for the Crank-Nicolson/Leap-Frog integration method", Technical report, Centrum Wiskunde and Informatica (CWI), $(2009).$

     \end{thebibliography}
     \end{document}